\documentclass[11pt]{amsart}
\usepackage{amsmath,amsthm,verbatim,amssymb,amsfonts,amscd,  graphicx}
\usepackage[normalem]{ulem}
\usepackage{xcolor}
\usepackage{graphics}

\usepackage{euscript, enumerate}
\usepackage{url,hyperref}

\usepackage{hyperref}

\usepackage{amscd}

\renewcommand{\thispagestyle}[1]{}
\usepackage{color}
\usepackage{geometry}
\newtheorem{thm}{Theorem}[section]
\newtheorem{lma}[thm]{Lemma}
\newtheorem{prop}[thm]{Proposition}
\newtheorem{cor}[thm]{Corollary}

\newtheorem*{thm*}{Theorem}
\newtheorem{definition}[thm]{Definition}
\newtheorem{remark}[thm]{Remark}

\newtheorem{Conj}[thm]{Conjecture}

\newcommand{\Ric}{\mbox{Ric}}
\newcommand{\R}{\mathbb R}
\newcommand{\Rm}{\mbox{Rm}}

\newcommand{\be}{\begin{equation}}
\newcommand{\ee}{\end{equation}}
\newcommand{\bee}{\begin{equation*}}
\newcommand{\eee}{\end{equation*}}

\def\D{\Delta_f}
\def\b{\beta\mathbb{}}
\def\na{\nabla}

\def\p{\partial}

\def\la{\langle}
\def\ra{\rangle}

\def\Pi{\displaystyle{\mathbb{II}}}

\def\e{\epsilon}
\def\ve{\varepsilon}

\def\a{\alpha}
\def\b{\beta}

\def\C{\mathbb{C}}

\begin{document}

\title{On a dichotomy of the curvature decay of steady Ricci soliton}
 \author{Pak-Yeung Chan}
\address[Pak-Yeung Chan]{Department of Mathematics,
University of California, San Diego,
La Jolla, CA 92093, USA.}
\email{pachan@ucsd.edu}

\author{Bo Zhu}
\address[Bo Zhu]{School of Mathematics,  University of Minnesota-Twin Cities, MN 55455, USA.}
\email{zhux0629@umn.edu}

\date{\today}

\renewcommand{\subjclassname}{
  \textup{2020} Mathematics Subject Classification}

\maketitle

\begin{abstract} We establish a dichotomy on the curvature decay for four dimensional complete noncompact non Ricci flat steady gradient Ricci soliton with linear curvature decay and proper potential function. A similar dichotomy is also shown in higher dimensions under the additional assumption that the Ricci curvature is nonnegative outside a compact subset.
\end{abstract}

\section{Introduction}
 Let $(M^n,g)$ be a smooth connected Riemannian manifold and $X$  a smooth vector field on $M$. The triple $(M,g,X)$ is said to be a Ricci soliton if there is a constant $\lambda \in \mathbb{R}$ such that
\be\label{eq-RS-1}
\Ric+\dfrac{1}{2}L_{X}g=\lambda g,
\ee
where $\Ric$ and $L_{X}$ denote the Ricci curvature and Lie derivative with respect to $X$ respectively. A Ricci soliton is called steady (shrinking, expanding) if $\lambda=0$ $(>0, <0 \text{  resp.})$. Upon scaling the metric by a constant, we may assume $\lambda \in \{-\frac{1}{2}, 0, \frac{1}{2}\}.$ The soliton is called complete if $(M,g)$ is complete as a Riemannian manifold. It is said to be gradient if $X$ can be chosen as $X=\nabla f$ for some smooth function $f$ on $M$. In this case, $f$ is called a potential function and (\ref{eq-RS-1}) can be rewritten as
\be\label{eq-RS-2}
\Ric+\nabla^2 f=\lambda g.
\ee
Ricci soliton is of great importance as it sometimes arises as a rescaled limit of the Ricci flow near its singularities. When $f$ in \eqref{eq-RS-2} is a constant, the metric becomes Einstein. Hence Ricci soliton can also be viewed as a natural generalization of the Einstein manifold as well.

In view of different examples of steady gradient Ricci solitons, the exponential and linear curvature decays seem to be two generic decays for noncompact steady solitons (see \cite{Chowetal-2007, Cao-2010} and ref. therein). 
Munteanu-Sung-Wang \cite{MunteanuSungWang-2017} raised the following conjecture on the curvature decay of steady solitons:
\begin{Conj}\label{dic conj}\cite{MunteanuSungWang-2017} If $(M, g, f)$ is a complete non Ricci flat steady gradient Ricci soliton with $|\Rm|\rightarrow 0$ as $x\to \infty$, then either one of the following estimates holds outside a compact set of $M$:
\be\label{linear decay conj}
C^{-1}r^{-1}\leq |\Rm|\leq Cr^{-1};
\ee
\be\label{exp decay conj}
C^{-1}e^{-r}\leq |\Rm|\leq Ce^{-r},
\ee
where $C$ is a positive constant and $r$ is the distance function.
\end{Conj}
The curvature decays \eqref{linear decay conj} and \eqref{exp decay conj} appear on the Bryant soliton and the Cigar soliton respectively (see \cite{Chowetal-2007} and ref. therein). This dichotomy conjecture, once established, will be very useful for the studies of steady soliton due to the classification results of steady solitons under
different curvature decay conditions (see \cite{Brendle-2013, CatinoMastroliaMonticelli-2016, DengZhu-2018, DengZhu-2019, DengZhu-2020.2, Deruelle-2012, MunteanuSungWang-2017, Chan-2019.2}). In fact, Deng-Zhu \cite{DengZhu-2018} (see also \cite{MunteanuSungWang-2017}) showed that the above dichotomy conjecture is true under nonnegative sectional curvature and linear scalar curvature decay conditions $R\leq C(r+1)^{-1}$. As an application, they classified $3$ dimensional steady gradient soliton with linear scalar curvature decay (see \cite{DengZhu-2018} and ref. therein).

\begin{thm}\cite{DengZhu-2018}\label{dichotomy in nonnegative sec} Suppose $(M^n, g, f)$ is a complete nonflat steady gradient Ricci soliton with nonnegative sectional curvature and linear scalar curvature decay, i.e. $R\leq C(r+1)^{-1}$. Then either \eqref{linear decay conj} or \eqref{exp decay conj} holds near infinity.
\end{thm}

Recently, Lai \cite{Lai-2020} has resolved a conjecture of Hamilton and constructed a family of $3$ dimensional gradient steady solitons that are flying wings. It is natural to examine their curvature properties to see if they fulfill the expectation from Conjecture \ref{dic conj}. As shown in \cite{Lai-2020}, $|\Rm|$ does not decay uniformly to $0$ at infinity on these $3$ dimensional examples. Analogous steady solitons were also constructed in dimension $n\geq 4$ \cite{Lai-2020}. Nonetheless, It is unclear at this point if the curvature $|\Rm|\to 0$ at infinity on these higher dimensional solitons. 

It follows from the Hamilton-Ivey pinching estimate that any $3$ dimensional complete steady gradient Ricci soliton must have nonnegative sectional curvature (see \cite{Chen-2009} and ref. therein). However, this significant feature does not hold in higher dimensions. In particular, the examples on some line bundles over the complex projective space constructed by Cao \cite{Cao-1996} and Appleton \cite{Appleton-2017} don't have nonnegative sectional curvature on the entire manifold. 
Hence, a reasonable next step is to consider if the dichotomy holds under linear curvature decay 
 condition in higher dimensions. Under the assumption that the soliton is $\kappa$ noncollapsed with dimension $n\geq 4$,  Deng-Zhu \cite{DengZhu-2018} obtained the estimate \eqref{lin decay in n dim} 
 if  $|\Rm|\leq C(r+1)^{-1}$ and $\Ric\geq 0$ outside a compact subset (see also \cite{DengZhu-2020.2, ChowDengMa-2020, BamlerChowDengMaZhang-2021}). However, there does exist collapsed steady solitons (see \cite{Cao-1996, DengZhu-2018.0, DengZhu-2020.3}). Munteanu-Sung-Wang \cite{MunteanuSungWang-2017} and the first named author \cite{Chan-2020} proved that the upper bound of $|\Rm|$ in \eqref{exp decay conj} holds, i.e.
 $$|\Rm|\leq Ce^{-r},$$
 if $r|\Rm|$ is uniformly small near infinity and $f$ is bounded from above. The lower bound in \eqref{exp decay conj} is a result by Chow-Lu-Yang \cite{ChowLuYang-2011} (see also \cite{MunteanuSungWang-2017}).
 
\bigskip
In this note, we are aimed at studying the curvature decay without any assumption of noncollapsedness and uniform smallness of $r|\Rm|$ at infinity. In particular, in real dimension $4$, if the scalar curvature $R$ decays at least linearly and the potential function $f$ is proper, then Conjecture \ref{dic conj} will hold. Our main theorem

\bigskip

\begin{thm}\label{4 dim dichotomy}
Let $(M^4, g, f)$ be a $4$ dimensional, complete, non Ricci flat steady gradient Ricci soliton with proper potential function $\lim_{x\to \infty} f=-\infty$ and linear scalar curvature decay, i.e. $R\leq C'(r+1)^{-1}$. Then there exists a positive constant $C$ such that either one of the following estimates holds near infinity
\be\label{4 dim linear decay rate}
C^{-1}r^{-1}\leq |\Rm|\leq Cr^{-1};
\ee
\be\label{4 dim exp decay rate}
C^{-1}e^{-r}\leq |\Rm|\leq Ce^{-r},
\ee
where $r$ is the distance function.
\end{thm}

\bigskip
If the Ricci curvature is nonnegative and goes to $0$ uniformly at infinity, then the potential function $f$ is proper \cite{CarrilloNi-2009}. Hence Theorem \ref{4 dim dichotomy} also generalizes
Theorem \ref{dichotomy in nonnegative sec} and the related results in \cite{DengZhu-2018, MunteanuSungWang-2017}.
Instead of using Gromov's almost flatness theorem \cite{Gromov-1978} or looking at the rescaled flow at infinity as in \cite{DengZhu-2018}. Here, we apply the parabolic maximum principle to the level set flow of $f$ as in \cite{Brendle-2014, MunteanuWang-2019} to obtain the dichotomy estimates. One advantages of this approach is that it doesn't require any volume noncollapsing condition. Similar argument works as well in higher dimensions under some assumptions on the Ricci curvature.

\bigskip 
\begin{thm}\label{n dim dichotomy}
Let $(M^n, g, f)$ be an $n\geq 5$ dimensional, complete, non Ricci flat steady gradient Ricci soliton with $\Ric\geq 0$ outside a compact subset and linear Riemann curvature decay, i.e. $|Rm|\leq C''(r+1)^{-1}$. Then there exists positive constant $C$ such that either one of the following estimates holds near infinity
\be\label{lin decay in n dim}
C^{-1}r^{-1}\leq R\leq c_n|\Rm|\leq Cr^{-1};
\ee
\be\label{exp in n dim}
C^{-1}e^{-r}\leq R\leq c_n|\Rm|\leq Ce^{-r},
\ee
where $c_n$ is some dimensional constant.
\end{thm}

\begin{remark}\label{Rm by R in n dim rmk} Under the linear curvature decay condition, we only need to assume that $|\Ric|\leq A_0 R$ for some constant $A_0>0$ and $f$ is proper, instead of $\Ric\geq 0$ near infinity, then the same conclusion in Theorem \ref{n dim dichotomy} will also hold. It can be seen from the volume estimate in \cite{Deruelle-2012} that when estimate \eqref{exp in n dim} is true, the soliton is collapsed. Hence Theorem $\ref{n dim dichotomy}$ also recovers an estimate in \cite{DengZhu-2018}. On the other hand, Theorem \ref{n dim dichotomy} tells us that
\be\label{Rm by R in n dim}
|\Rm|\leq cR \text{  on  } M.
\ee
See \cite{Chan-2019} and references therein for the case of $n=4$. Ancient solution to the Ricci flow and steady soliton satisfying estimate \eqref{Rm by R in n dim} have been recently studied by Ma-Zhang \cite{MaZhang-2021} and the first named author-Ma-Zhang \cite{ChanMaZhang-2021}.
\end{remark}

\bigskip 

For the  level set of potential function under the dichotomy, we have the following two topological classifications of level set for large $\tau$.
\begin{itemize}
    \item If we assume that \eqref{4 dim linear decay rate} or \eqref{lin decay in n dim} holds, then by the Gauss equation 
and Shi's estimate \eqref{Shi est}, the level set $\Sigma_{\tau}:=\{-f=\tau\}$ must have positive scalar curvature for large $\tau$;
\item  If we assume that \eqref{4 dim exp decay rate} or \eqref{exp in n dim} holds, then it follows from Proposition \ref{level sets are e flat} and Shoen-Yau's trick on minimal surface that the level set is diffeomorphic to a compact flat manifold (see also \cite{Deruelle-2012}). 
\end{itemize}
Hence, as a corollary, we obtain that

\begin{cor}\label{level set dich}Under the same assumptions as in Theorem \ref{4 dim dichotomy} or Theorem \ref{n dim dichotomy}, there exists a large constant $\tau_0$ such that either one of the following is true for the compact connected hypersurface $\Sigma_{\tau}:=\{-f=\tau\}$:
\begin{enumerate}
   \item $\Sigma_{\tau}$ is diffeomorphic to a finite quotient of torus for all $\tau\geq \tau_0$;
   \item $\Sigma_{\tau}$ has positive scalar curvature w.r.t. the intrinsic metric induced by $(M,g)$ for all $\tau\geq \tau_0$.
\end{enumerate}
\end{cor}

\begin{remark}
Under the same conditions as in Theorem \ref{n dim dichotomy} (with $n \geq 4$ instead) and additionally the soliton is $\kappa$ noncollapsed, Deng-Zhu \cite{DengZhu-2018} proved that that the level sets of $f$ at infinity are diffeomorphic to a compact gradient Ricci shrinker with nonnegative Ricci curvature. In particular for as $n=4$, the level sets at infinity are diffeomorphic to a spherical $3$ manifolds (see also \cite{ChowDengMa-2020}). However, generally it may not be the case if noncollapsed condition is not assumed.
For instance, let $B_3$ be the three dimensional Bryant steady soliton, the level sets of $f$ at infinity of the product soliton $B_3\times \mathbb{S}^1$ are diffeomorphic to $\mathbb{S}^2\times \mathbb{S}^1$ which has infinite fundamental group $\pi_1$ and admits no shrinker structure \cite{FM-2008}. However, it does support a metric of positive scalar curvature as in Corollary \ref{level set dich} $(b)$. Here, 
in the case of positive scalar curvature
\begin{itemize}
    \item As $n=4$, by the resolutions of the Poinc\'{a}re Conjecture and the Thurston’s Geometrization Conjecture \cite{p1,p2,p3}, $\Sigma_\tau$ is a connected sum of spherical 3-manifolds and some copies of $S^1\times S^2$;
\item 
As $n\geq 5$, by the $\Ric\geq 0$ near infinity condition, the level sets $\Sigma_{\tau}$ in Corollary \ref{level set dich}$(b)$ have almost nonnegative Ricci curvature and hence have first Betti number $b_1(\Sigma_{\tau})\leq n-2$ by a result of Cheeger-Colding \cite[Theorem A.1.13]{CheegerColding-1997}, where $n$ is the real dimension of the ambient manifold $M$.
\end{itemize}

\end{remark}

\bigskip
Dimension reduction serves as an important tool in the studies of Ricci flow and has led to a lot of successes in the classification of Ricci solitons (for example, see \cite{BamlerChowDengMaZhang-2021, Brendle-2013, Brendle-2014, ChowDengMa-2020, ChowLuNi-2006, DengZhu-2018.0, DengZhu-2018, DengZhu-2019, DengZhu-2020.2, Hamilton-1995, MunteanuWang-2019} and ref. therein). Under some volume noncollapsing conditions, one may extract different geometric information by looking at the blow down profiles of the soliton at infinity, which usually split like $\R\times X$, where $X$ is some ancient solution to the Ricci flow of lower dimension. On $4$ dimensional complete noncompact, noncollapsed, non Ricci flat steady gradient soliton with $R\to 0$ at $\infty$ and $\Ric\geq 0$ outside compact set, Chow-Deng-Ma \cite{ChowDengMa-2020} studied the possible rescaled limits and showed that the limits at $\infty$ split off a line. In general, existence of smooth limit at infinity is not expected without the noncollapsedness condition. In such a case, Gromov-Hausdorff topology is one of the most natural notions to work on for the convergence of manifolds. Under the conditions of Theorems \ref{4 dim dichotomy} and \ref{n dim dichotomy}, we apply the level set method in \cite{DengZhu-2018, DengZhu-2019, DengZhu-2020.2} 
to show that the Gromov-Hausdorff limit of the blow down metrics based at any sequence $p_i\to\infty$ also splits isometrically. Very recently, Bamler \cite{Bamler-2020.0, Bamler-2020, Bamler-2020.1} has developed new compactness and partial regularity theories of Ricci flow under relatively weak conditions. Bamler-Chow-Deng-Ma-Zhang \cite{BamlerChowDengMaZhang-2021} also classified the tangent flows at infinity (see \cite{Bamler-2020} for the definition) 
 of $4$ dimensional steady soliton singularity models.

Let $(N, h)$ be a complete Riemannian manifold and $L$ the cylinder $\R\times N$ with product metric $g_L=dt^2+h$. For any $\lambda\, \in\R$, the translation $\rho_{\lambda}: L\longrightarrow L$ is given by $\rho_{\lambda}(s,\omega):=(s+\lambda,\omega)$. A Riemannian manifold $(M,g)$ is said to be smoothly asymptotic to the cylinder $L$ if there exist a compact set $K$ of $M$ and a diffeomorphism $\Phi:(0,\infty)\times N \longrightarrow M\setminus K$ such that $\rho_{\lambda}^*\Phi^*g\longrightarrow g_L$ in $C^{\infty}_{loc}$ sense on $((0,\infty)\times N, g_L)$ as $\lambda \to \infty$ (see \cite{MunteanuWang-2019}).
The asymptotic convergence is at exponential rate if in addition for any integer $k\geq 0$, there is a positive constant $C_k$ such that for all $s> 0$,
\be
\sup_{\omega\in N}\left|\na_{g_L}^k\left(\Phi^*g-g_L\right)\right|_{g_L}(s, \omega)\leq C_k e^{-s}.
\ee

\begin{cor}\label{GH limit}
Under the same assumptions as in Theorem \ref{4 dim dichotomy} or Theorem \ref{n dim dichotomy}, either one of the following holds:
\begin{enumerate}
   \item For any sequence $p_i\to \infty$ in $M$, after passing to a subsequence, $(M, d_{R(p_i)g}, p_i)$ converges in pointed Gromov-Hausdorff sense to a cylinder $(\R\times Y, \sqrt{d_e^2+d_
   Y^2}, p_{\infty})$, where $d_e$ is the flat metric on $\R$, $(Y,d_Y)$ denotes a compact Alexandrov space 
   and $\sqrt{d_e^2+d_
   Y^2}$ indicates the product metric (see also \cite{CheegerColding-1996}).
   \item For any sequence $p_i\to \infty$ in $M$, $(M, d_{R(p_i)g}, p_i)$ converges in pointed Gromov-Hausdorff sense (without passing to subsequence) to the ray $([0,\infty), d_e, 0)$,  where $d_e$ is the flat metric restricted on $[0, \infty)$. In this case, $(M,g)$ is smoothly asymptotic to the flat cylinder $\R\times\left(\mathbb{T}^{n-1}\big/\sim\right)$ 
   at exponential rate, where $\mathbb{T}^{n-1}\big/\sim$ is diffeomorphic to the quotient of torus in Corollary \ref{level set dich}$(a)$.
   \end{enumerate}
\end{cor}

\bigskip 
Since we assume $\Ric\geq 0$ outside a compact subset when $n\geq 5$, so in this case, the result in Corollary \ref{GH limit}$(a)$ also follows from the splitting theorem by Cheeger-Colding \cite{CheegerColding-1996}. 
Instead, we shall adopt the level set approach by Deng-Zhu \cite{DengZhu-2018, DengZhu-2019, DengZhu-2020.2} which works for all $n\geq 4$ and provides more geometric information of the metric space $(Y, d_Y)$ (see also Proposition \ref{GH limit2}). In particular, it can be seen from the proof that $(Y, d_Y)$ in Corollary \ref{GH limit}$(a)$ is the Gromov Hausdorff limit of a sequence of level sets as in Corollary \ref{level set dich}$(b)$ with uniformly bounded curvature and diameter after scaling. 
In general, due to volume collapsing, the limit $\R\times Y$ in Corollary \ref{GH limit}$(a)$ can be of lower dimension compared to $M$. On the positively curved Cao steady K\"{a}hler soliton on $\C^2$ \cite{Cao-1996}, the corresponding metric space $(Y, d_Y)$ is $\mathbb{CP}^{1}$ endorsed with a scalar multiple of the Fubini-Study metric (see \cite{DengZhu-2018.0}).

\bigskip 
Steady Ricci solitons with fast curvature decay were extensively studied in \cite{CatinoMastroliaMonticelli-2016, Chan-2019.2, DengZhu-2015, DengZhu-2018, Deruelle-2012, MunteanuSungWang-2017}. As an application of Theorems \ref{4 dim dichotomy} and \ref{n dim dichotomy}, by exploiting the real analyticity of the soliton metric, we prove some classification results on steady gradient Ricci soliton with fast curvature decay under milder conditions. Here, we do not require any global non-negative curvature condition and uniform smallness of the quantity $rR$ at infinity. 


\begin{thm}\label{fast decay soliton}
Suppose $(M^m, g, f)$ is a complete, non Ricci flat steady K\"{a}hler gradient Ricci soliton of complex $m$ dimension with nonnegative Ricci curvature outside a compact subset of $M$, where $m\geq 2$, such that
\be
\liminf_{x\to\infty}rR=0.
\ee 
Further assume that the following conditions are satisfied,
\begin{itemize} 
\item $R\leq C(r+1)^{-1}$ if $m=2$;
\item $|\Rm|\leq C(r+1)^{-1}$ if $m\geq3$.
\end{itemize}
Then $M$ is holomorphically isometric to a quotient of $\Sigma\times \mathbb{C}^{m-1}$ and the curvature $\Rm$ decays exponentially in $r$, where $\Sigma$ is the Hamilton Cigar soliton.
\end{thm}

\begin{thm}\label{fast decay Rm soliton}
Let $(M^n, g, f)$ be a complete noncompact and nonflat steady gradient Ricci soliton with nonnegative sectional curvature outside a compact subset of $M$ and linear scalar curvature decay, that is, $R\leq C(r+1)^{-1}$. If in addition that
\be
\liminf_{x\to\infty} rR=0,
\ee
then M is isometric to a quotient of $\Sigma\times \R^{n-2}$, where $\Sigma$ is the Hamilton Cigar soliton.
\end{thm}
\begin{remark}
If we further assume sectional curvature is nonnegative everywhere on $M$, then Theorem \ref{fast decay Rm soliton} also follows 
from \cite{DengZhu-2018} (see also Theorem \ref{dichotomy in nonnegative sec} and \cite{Chan-2019.2, MunteanuSungWang-2017}).
\end{remark}

We wrap up the introduction by looking at the analogous dichotomy in other types of Ricci solitons.
For complete noncompact and nonflat gradient Ricci shrinker, Munteanu-Wang \cite{MunteanuWang-2017} showed that If $|\Ric|$ is sufficiently small at infinity, then the shrinker is smoothly asymptotic to a cone and hence $|\Rm|\sim r^{-2}$ at infinity by a result of Chow-Lu-Yang \cite{ChowLuYang-2011} (see also \cite{KotschwarWang-2015}). In real dimension $4$, it remains an open problem whether shrinker with bounded scalar curvature $R$ must have either $R\geq c$ for some positive constant $c$ or $R\to 0$ at infinity (this holds in the K\"{a}hler case \cite{MunteanuWang-2019}).  The dichotomy will be extremely useful toward the classification of $4$ dimensional shrinker in view of different studies on the asymptotic geometry of the solitons (see \cite{KotschwarWang-2015, MunteanuWang-2017, MunteanuWang-2019} and ref. therein).

\bigskip 

Before we move on, let us recall the notion of asymptotically conical expanding soliton \cite{Deruelle-2017}. Let $X$ be a smooth $n-1$ dimensional closed manifold with Riemannian metric $g_X$, $C(X)$ is defined to be the cone over $X$, i.e. $\{(t,\omega):$  $t>0, \omega \in X\}$. $g_C$ and $\na_C$ denote the metric $dt^2+t^2g_X$ on $C(X)$ and its Levi-Civita connection respectively. For any positive constant $S$, $\overline{B(o, S)}\subseteq C(X)$ is the set given by $\{(t,\omega):$  $S\geq t>0,\, \omega \in X\}$.

\begin{definition}\label{AC expander} \cite{Deruelle-2016, Deruelle-2017} A complete expanding gradient Ricci soliton $(M, g, f)$ is asymptotically conical (at polynomial rate $\tau=2$)  with asymptotic cone $(C(X), g_C)$ if there exist constants $S_0>0$ and $c_0$, a compact set $K$ in $M$ and a diffeomorphism $\phi:$ $M\setminus K$ $\rightarrow$ $C(X)\setminus \overline{B(o, S_0)}$ such that for any nonnegative integer $k$ \
\be\label{cone derivative estimates for metric}
\sup_{\omega\,\in X}\left|\na_C^k \left[(\phi^{-1})^*g-g_C\right]\right|_{g_C}(t,\omega)=O(t^{-2-k}) \text{  as  } t \to \infty ;
\ee
\be\label{f in cone}
f\circ \phi^{-1}(t,\omega)=-\frac{t^2}{4}+c_0 \text{  for all } t> S_0.
\ee

\end{definition}

Just like the steady case, there are two generic curvature decays at infinity for nonflat asymptotically conical expander \cite{Deruelle-2017}:
\be\label{quadratic decay for expander}
C^{-1}v^{-1}\leq |\Rm| \leq Cv^{-1};
\ee
\be\label{exp decay for expander}
 C^{-1}v^{1-\frac{n}{2}}e^{-v}\leq |\Rm|\leq Cv^{1-\frac{n}{2}}e^{-v},
\ee
where $v=\frac{n}{2}-f$ and $\lim_{x\to\infty}4r^{-2}v=1$.  The curvature decays \eqref{quadratic decay for expander} and \eqref{exp decay for expander} can be found in the Bryant expanding soliton and the K\"{a}hler expander constructed by Feldman-Ilmanen-Knopf \cite{FeldmanIlmanenKnopf} respectively. It will be interesting to see whether or not \eqref{quadratic decay for expander} and \eqref{exp decay for expander} are the only possible curvature decays of expanding soliton. The upper bound in \eqref{quadratic decay for expander} is always satisfied on conical expander. If $\lim_{x\to \infty} r^2|\Rm|=0$, then the upper and lower estimates of $\Rm$ in \eqref{exp decay for expander} are due to Deruelle \cite{Deruelle-2017} and the first named author \cite{Chan-2020} respectively. Dichotomy result in the expanding case similar to Theorems \ref{4 dim dichotomy} and \ref{n dim dichotomy} seems to be quite promising. However using the existence and compactness results for Ricci expander by Deruelle \cite{Deruelle-2016, Deruelle-2017}, one can prove that the above dichotomy fails in $3$ dimensional expanding case: 
\begin{thm}\cite{Deruelle-2016, Deruelle-2017}\label{counter eg expander}
There exists a $3$ dimensional complete noncompact asymptotically conical expanding gradient Ricci soliton $(M, g, f)$ with nonnegative curvature operator $\Rm\geq 0$ and 
\be\label{failure}
0=\liminf_{x\to\infty} v|\Rm|<\limsup_{x\to \infty}v|\Rm|<\infty,
\ee
where $v=\frac{n}{2}-f$ and $\lim_{x\to\infty}4r^{-2}v=1$.
\end{thm}

\begin{remark}
Since $0=\liminf_{x\to\infty} v|\Rm|$, $M$ doesn't satisfy the lower bound in \eqref{quadratic decay for expander}. The upper estimate in \eqref{exp decay for expander} doesn't hold as $\limsup_{x\to \infty}v|\Rm|>0$.
\end{remark}
Though Theorem \ref{counter eg expander} is not explicitly stated in \cite{Deruelle-2016, Deruelle-2017}, it is essentially due to Deruelle and is a direct consequence of the results in \cite{Deruelle-2016, Deruelle-2017}. We shall include its proof in the appendix for the sake of completeness.

\bigskip

The paper is organized as follows. In Section 2, we include the basic preliminaries of steady soliton. In Section 3, we show the exponential curvature decay and prove Theorems \ref{4 dim dichotomy} and \ref{n dim dichotomy}. Sections 4 and 5 will then be devoted to the proofs of Theorems \ref{fast decay soliton} and \ref{fast decay Rm soliton}. In Section 6, we will prove Corollary \ref{GH limit}. Finally, the proof of Theorem \ref{counter eg expander} will be presented in the appendix.\\

{\sl Acknowledgement}: The authors are indebted to Jiaping Wang for helpful discussions and suggestions. The authors would also like to thank Bennett Chow, Zilu Ma, Ovidiu Munteanu and Yongjia Zhang for valuable comments on this work. The first named author was partially supported by an AMS–Simons Travel Grant and would like to express his gratitude to Lei Ni for constant encouragement.

\section{Preliminaries}
Let $(M,g)$ be an $n$ dimensional connected smooth Riemannian manifold and $f$ be a smooth function on $M$. $\nabla$ is the Levi Civita connection of $g$. $\Rm$, $\Ric$ and $R$ denote the Riemann, Ricci and scalar curvatures respectively. $(M,g,f)$ is said to be a gradient steady Ricci soliton with the potential function $f$ if
\be\label{eq-RS-22}
\Ric+\na^2 f=0.
\ee
The steady soliton is complete if $(M,g)$ is a complete Riemannian manifold. $(M,g,f)$ is a gradient steady K\"{a}hler Ricci soliton if $M$ satisfies (\ref{eq-RS-22}) and is a complex manifold with complex structure $J$ and $g$ is a K\"{a}hler metric compatible with $J$ on $M$. 
Ricci soliton is a self similar solution to the Ricci flow. Given a complete steady gradient Ricci soliton, we consider $g(t):=\varphi_t^*g$, where $t$ $\in \R$ and $\varphi_t$ is the flow of $\na f$ with  $\varphi_0= id$. Then $g(t)$ is an ancient solution to the Ricci flow: 
\be\label{eq-RF-1}
\begin{split}
\dfrac{\partial g(t)}{\partial t}&=-2\Ric(g(t)),\\
g(0)&=g.
\end{split}
\ee

It was shown by Chen \cite{Chen-2009} that any complete ancient solution to the Ricci flow must have nonnegative scalar curvature (see also \cite{Zhang-2009}). Using the strong maximum principle, we see that any complete steady gradient Ricci soliton must have positive scalar curvature $R>0$ unless it is Ricci flat. It is also known that compact steady soliton must be Ricci flat \cite{Chowetal-2007}. Upon scaling the metric by a constant if necessary, Hamilton \cite{Hamilton-1995} showed that for a complete steady gradient Ricci soliton,
\be\label{conserved eqn}
|\na f|^2+R=1 \text{  on  } M,
\ee
we adopt the above scaling convention throughout this article. By virtue of \eqref{eq-RS-22} and the Ricci identity, 
for any vectors $X, Y$ and $Z$
\be\label{Ric identity}
\na_Z\Ric(X,Y)-\na_Y\Ric(X,Z)=R(Z,Y,X,\na f).
\ee
Hereinafter, we fix a point $p_0$ in $M$, then for any $x$ in $M$, $r$, $r(x)$ and $d(x, p_0)$ will be used interchangeably and refer to the distance between $x$ and $p_0$. 

Then, for any smooth function $\omega$, we define the weighted Laplacian w.r.t. $\omega$ by $\Delta_{\omega}:=\Delta- \na \omega\cdot \na $, where $\Delta$ is Beltrami Laplacian operator. The following identities are well known for steady gradient Ricci soliton (see \cite{Chowetal-2007})
\begin{eqnarray}
\label{f lap v}\D v&=&1\\
\Delta f+R&=&0\\
\label{na R=Ric}\na R&=&2\Ric(\na f)\\
\label{f lap R}\D R&=&-2|\Ric|^2,
\end{eqnarray}
where $v=-f$. When $f$ is proper, i.e. 
$$\lim_{x\to\infty}f(x)=-\infty,$$
by adding a constant to $f$ if necessary, we may assume that $v=-f\geq 10$ on $M$. Furthermore by \cite[Lemma 2]{Chan-2019}, $f$ and $v$ must be distance-like and satisfy
\be\label{pot est vs r}
\lim_{x\to\infty}\frac{v}{r}=-\lim_{x\to\infty}\frac{f}{r}=1.
\ee
Since we assume that the scalar curvature $R\longrightarrow 0$ as $x\to\infty$, $\lim_{x\to\infty}|\na f|^2=\lim_{x\to\infty}1-R=1$ and the level sets $\Sigma_{\tau}:=\{-f=\tau\}$ are smooth compact hypersurfaces for all large $\tau$. Moreover, $\Sigma_{\tau}$ are connected by a result of Munteanu-Wang \cite{MunteanuWang-2011}.




\section{Curvature decay}
The goal of this section is to prove the following key proposition
\begin{prop}\label{level sets exp decay}
Let $(M^n,g, f)$ be a complete noncompact non Ricci flat steady gradient Ricci soliton of real dimension $n$ with proper potential function. Further suppose that the following conditions are satisfied:
\begin{itemize} 
\item $|\Ric|\leq A_0R$ on $M$ for some constant $A_0>0$;
\item $|\Rm|\leq A_1(r+1)^{-1}$ on $M$ for some constant $A_1>0$;
\item $\liminf_{x\to\infty}rR=0$.
\end{itemize}
Then for some constant $C>0$
\be
|\Rm|\leq Ce^{-r}.
\ee
\end{prop}

\vskip 2mm


We will make use of the diameter and curvature estimates by  Deng-Zhu \cite{DengZhu-2018, DengZhu-2020.2}. By viewing the soliton as a solution to the Ricci flow and our assumption $|\Rm|\leq C(r+1)^{-1}$ implies, by the Shi's estimate \cite{ChowLuNi-2006},  that for any nonnegative integer $k$, there exists positive constant $C_k$ such that on $M$
\be\label{Shi est}
|\nabla^k\Rm|\leq C_k(r+1)^{-\left(\frac{k+2}{2}\right)}.
\ee
Using \eqref{Shi est}, we may argue in the same way as in \cite{DengZhu-2018, DengZhu-2020.2} to get the following lemma
\begin{lma}\cite{DengZhu-2018, DengZhu-2020.2}\label{level set diam est} If $(M^n,g,f)$ is a complete noncompact steady gradient Ricci soliton of real dimension $n$ with proper potential function $f$ and linear Riemann curvature decay $|\Rm|\leq A_1(r+1)^{-1}$,
then there exists constant $C>0$ such that for all sufficiently large $\tau$
\be
\text{diam}(\Sigma_{\tau})\leq C\sqrt{\tau},
\ee
where $\text{diam}(\Sigma_{\tau})$ denote the intrinsic diameter of the level set $\Sigma_{\tau}=\{-f=\tau\}$ w.r.t. induced metric by $g$.
\end{lma}

With the diameter estimate on the level sets, we may proceed as in \cite{DengZhu-2018} to show the following lemma:
\begin{lma}\cite{DengZhu-2018}\label{local fast conv} Let $(M^n,g,f)$ be a complete noncompact steady gradient Ricci soliton of real dimension $n$ with proper potential function $f$. Further suppose that 
the curvature tensor $\Rm$ decays at least linearly i.e. 
\be\label{lin Rm decay cond in ball}
|\Rm|\leq A_1(r+1)^{-1}.
\ee
If in addition, $r(p_i)R(p_i) \to 0$ for a sequence of points $p_i\to \infty$, then there exists a positive constant $\varepsilon (n, A_1)$ such that 
\be
\sup_{B(p_i,\varepsilon; r(p_i)^{-1}g)}rR \longrightarrow 0 \text{  as } i \to \infty,
\ee
where $B(p_i,\varepsilon; r(p_i)^{-1}g)$ is the geodesic ball on $M$ centered at $p_i$ with radius $\varepsilon$ w.r.t. metric $r(p_i)^{-1}g$.
\end{lma}

\begin{proof}
The argument is due to Deng-Zhu \cite{DengZhu-2018} and we  include the proof for the sake of completeness. Let $r_i=r(p_i)$ and $g_i=r_i^{-1}g$.
We consider the Ricci flow associated to the steady soliton, i.e. $g(t):=\phi_t^*g$, $\phi_t$ is the flow of the vector field $\na f$ with $\phi_0=\text{  id}$ and $t$ $\in \R$. For any $z$ $\in B(p_i, 1; g_i)=B(p_i, \sqrt{r_i}; g)$, by the triangle inequality
\be\label{dist est for conv}
r_i+\sqrt{r_i}\geq r(z) \geq r_i-\sqrt{r_i}.
\ee
Moreover, using $|\na f|\leq 1$ (see \eqref{conserved eqn}), for $t$ $\in [-\frac{1}{10}, \frac{1}{10}]$, we have $d_g(\phi_{r_it}(z), z)\leq \frac{r_i}{10}$ and 
\be\begin{split}
r(\phi_{r_it}(z))&\geq r(z)-\frac{r_i}{10}\\
&\geq \frac{r_i}{2}.
\end{split}
\ee
Let $g_i(t):=\frac{1}{r_i}g(r_it)$, then $g_i(0)=g_i$ and by \eqref{lin Rm decay cond in ball}, on $B(p_i, 1; g_i)\times [-\frac{1}{10}, \frac{1}{10}]$
\be\label{C0 est of flow}\begin{split}
|\Rm(g_i(t))|(z)&=r_i|\Rm(g(r_it))|(z)\\
&=r_i|\Rm(g)|(\phi_{r_it}(z))\\
&\leq \frac{A_1r_i}{r(\phi_{r_it}(z))}\\
&\leq 2A_1.
\end{split}
\ee
Thank to the Shi's estimates \eqref{Shi est}, for all positive integer $k$
\be\label{Ck est of flow}\begin{split}
|\na^k\Rm(g_i(t))|(z)&=r_i^{\frac{k+2}{2}}|\na ^k\Rm(g(r_it))|(z)\\
&=r_i^{\frac{k+2}{2}}|\na ^k\Rm(g)|(\phi_{r_it}(z))\\
&\leq C_k2^{\frac{k+2}{2}}.
\end{split}
\ee
We are about to take subsequential limit of $g_i(t)$. Let $\eta_i: (\R^n, g_e)\longrightarrow (T_{p_i}M, g_i(p_i))$ be a linear isometry with $\eta(0)=0$. By the Rauch Comparison theorem, the map $F_i:=\exp^{g_i(0)}_{p_i}\circ\, \eta_i$ is a local diffeomorphism on $\{x\in \R^n:$ $|x|<\min{\left(\frac{1}{2}, \frac{\pi}{\sqrt{2A_1}}\right)} \}$, where $\exp^{g_i(0)}_{p_i}$ is the exponential map at $p_i$ w.r.t. $g_i(0)$. By the Hamilton compactness theorem, \eqref{C0 est of flow} and \eqref{Ck est of flow} (see also \cite{Hamilton-1995.2}, \cite{Chowetal-2007}, \cite[Lemma 4.4]{DengZhu-2018} and \cite[Theorem 12]{Chan-2020}), there exist positive constant $\delta=\delta(n, A_1)<\min{\left\{\frac{1}{2}, \frac{\pi}{\sqrt{2A_1}}\right\}}$ and a subsequence $i_l$ such that as $l\to\infty$
\be\label{local conv of flow}
\big(B_{\delta}(0), F_{i_l}^*g_{i_l}(t)\big)\longrightarrow \big(B_{\delta}(0), g_{\infty}(t)\big)
\ee
in $C^{\infty}_{loc}$ sense on $B_{\delta}(0)\times (-\frac{1}{10}, \frac{1}{10})$, where $B_{\delta}(0):=\{x\in\R^n: |x|<\delta\}$ and $g_{\infty}(t)$ is a solution to the Ricci flow. Note that $\delta$ is independent on the subsequence taken. $F_{i_l}^*g_{i_l}(t)$ has nonnegative scalar curvature, so does $g_{\infty}(t)$. Furthermore by the assumption and smooth convergence
\be\begin{split}
R_{g_{\infty}(0)}(0)&= \lim_{l\to\infty}R_{F_{i_l}^*g_{i_l}(0)}(0)\\
&=\lim_{l\to\infty}r_{i_l}R_{g}(F_{i_l}(0))\\
&=\lim_{l\to\infty}r_{i_l}R_{g}(p_{i_l})\\
&= 0.
\end{split}
\ee
Due to the strong minimum principle,  $R_{g_{\infty}(t)}\equiv 0$ on $B_{\delta}(0)\times (-\frac{1}{10}, 0]$. 
For the above $\delta$, we want to show that the original sequence $p_i$ satisfies
\be
\lim_{i\to\infty}\sup_{B(p_i, \delta/2; g_i)}rR=0,
\ee
Suppose on the contrary, by passing to a subsequence if necessary, we may assume there is constant $\e_0>0$ such that for all $i$
\be\label{contrad}
\sup_{B(p_i, \delta/2; g_i)}rR\geq \e_0.
\ee
By the locally uniform convergence \eqref{local conv of flow} and distance estimate \eqref{dist est for conv}
\be\begin{split}
o(1)&=\sup_{z\in B_{\frac{\delta}{2}}(0)}R_{F_{i_l}^*g_{i_l}(0)}(z)\\
&\geq \sup_{y\in B(p_{i_l}, \frac{\delta}{2}; g_i)}R_{g_{i_l}(0)}(y)\\
&= \sup_{y\in B(p_{i_l}, \frac{\delta}{2}; g_i)}r_{i_l}R_{g}(y)\\
&\geq (1-o(1))\sup_{y\in B(p_{i_l}, \frac{\delta}{2}; g_i)}r(y)R_{g}(y)\\
&\geq (1-o(1)) \e_0,
\end{split}
\ee
which is impossible. This completes the proof of the lemma.
\end{proof}

Using Lemma \ref{local fast conv}, Deng-Zhu \cite{DengZhu-2018} can show a better convergence result.
\begin{lma}\cite{DengZhu-2018}\label{local fast conv2} Under the same notations and assumptions as in Lemma \ref{local fast conv}, we have \be
\sup_{\Sigma_{-f(p_i)}}rR \longrightarrow 0 \text{  as } i \to \infty.
\ee
\end{lma}

\begin{proof} The proof is due to Deng-Zhu \cite{DengZhu-2018} and we provide a sketch of the argument for reader's convenience. We argue by contradiction, by passing to subsequence if necessary, there exists constant $\e_0>0$ such that for all $i$, we can find $q_i$ $\in \Sigma_{-f(p_i)}$ such that 
\be
r(q_i)R(q_i)\geq \e_0.
\ee
Let $\gamma_i$ be a normalized minimizing intrinsic geodesic joining $p_i$ to $q_i$ on $\Sigma_{-f(p_i)}$ with respect to induced metric $g_i=r(p_i)^{-1}g$, $l(\gamma_i)$ the length of $\gamma_i$ with respect to $g_i$. Then  $\frac{N_i\e}{8}\leq l(\gamma_i)<\frac{(N_i+1)\e}{8}$ for some unique nonnegative integer $N_i$, where $\ve$ is the constant as in Lemma \ref{local fast conv}. Moreover by \eqref{pot est vs r}, for all large $i$ and $a$ $\in \Sigma_{-f(p_i)}$
\be\label{tau eqv r}
\frac{1}{2} r(a)\leq -f(a)\leq 2r(a).
\ee
In view of Lemma \ref{level set diam est}, we have the following upper bound for $N_i$ 
\be
N_i\leq 16C\e^{-1}.
\ee
Hence by taking further subsequence, we may assume $N_i\equiv N_1$ for all $i$. For each $0\leq j\leq N_1$, we define the sequence $p^{j}_i:=\gamma_i\left(\frac{j\e}{8}\right) $. Using \eqref{tau eqv r}, We see that $p^{0}_j=p_i$ and the distance function on $M$  satisfies
\be\label{dist est for whole level set conv}
\begin{split}
d_{r(p_i^j)^{-1}g}(p^{j}_i, p^{j+1}_i)=\sqrt{\frac{r(p_i)}{r(p_i^j)}}\, d_{g_i}(p^{j}_i, p^{j+1}_i)\leq 2d_{g_i}(p^{j}_i, p^{j+1}_i)\leq \frac{\ve}{4}
\end{split}
\ee
and similarly
\be\label{dist est for whole level set conv2}
d_{r(p_i^{N_1})^{-1}g}(p^{N_1}_i, q_i)\leq \frac{\ve}{4}.
\ee
By \eqref{dist est for whole level set conv} and Lemma \ref{local fast conv}, $r(p_i)R(p_i)\to 0$  would imply  $r(p^1_i)R(p^1_i)\to 0$. Applying Lemma \ref{local fast conv} again with $p_i$ replaced by $p^1_i$. we see that  $r(p^2_i)R(p^2_i)\to 0$. Similarly by \eqref{dist est for whole level set conv2} and repeating the same procedure finitely many times, we conclude that
\be
\begin{split}
0&= \lim_{i\to\infty}r(q_i)R(q_i)\\
&\geq \e_0,
\end{split}
\ee 
which is absurd. Result follows.

\end{proof}

To apply the parabolic maximum principle as in \cite{Brendle-2014, MunteanuWang-2019}, we need the evolution equation of $vR$ along the level set flow:
\begin{lma}\label{evolution of vR}
Suppose that $(M^n, g, f)$ is a complete non Ricci flat steady gradient Ricci soliton with proper potential function $\lim_{x\to\infty}f=-\infty$ and linear Riemann curvature decay $|\Rm|\leq A_1(r+1)^{-1}$ for some constant $A_1>0$. Then there exist positive constants $a_0$ and $C_0$ such that on $\{x\in M:$ $ v(x)\geq a_0\}$, $|\na f|^2\geq \frac{1}{2}$ and 
\be
\Delta_{\Sigma_{\tau}}(vR)-\la \na f, \na (vR)\ra \geq -2v|\Ric|^2+R-C_0v^{-\frac{3}{2}},
\ee
where $v:=-f$ and $\Delta_{\Sigma_{\tau}}$ denotes the intrinsic Laplacian of the level set $\Sigma_{\tau}:=\{-f=\tau\}$ with respect to the induced metric by $g$.
\end{lma}
\begin{proof} By \eqref{conserved eqn}, $|\na f|^2=1-R=1-o(1)\geq \frac{1}{2}$ outside a compact set. Let $\{e_i\}_{i=1}^{n}$ be an orthonormal frame near infinity such that $e_n=\frac{\na f}{|\na f|}$. It is known that the Laplace operators on $M$ and $\Sigma_{\tau}$ are related by the following formula: for any smooth function $\omega$,
\be
\Delta \omega=\na^2\omega(e_n, e_n)+H_\tau  e_n \omega + \Delta_{\Sigma_{\tau}}\omega.
\ee
where $\Delta$ and $\Delta_{\Sigma_\tau}$ are the Laplacian operators on $M$ and $\Sigma_{\tau}$ respectively, $H_\tau$ is the mean curvature of $\Sigma_{\tau}$.
From the soliton equation, we see that $\la e_i,\na_{e_i} e_n\ra=-\Ric(e_i, e_i)|\na f|^{-1}$ for $1\leq i\leq n-1$ and $H_\tau=\sum_{i=1}^{n-1} \la\na_{e_i}e_n, e_i\ra$. Hence
\be\label{lap relation}
\Delta \omega= \Delta_{\Sigma_{\tau}}\omega+\frac{ \na^2 \omega (\na f, \na f)}{|\na f|^2}-\frac{R-\Ric(\frac{\na f}{|\na f|}, \frac{\na f}{|\na f|})}{|\na f|^{2}}\la \na f, \na \omega\ra.
\ee
We substitute $\omega=vR$ in \eqref{lap relation} and estimate the terms on the R.H.S. of \eqref{lap relation} one by one. Using \eqref{Shi est}, \eqref{pot est vs r} and 
\be\label{Ric in na f}
2\Ric(\na f, \na f)=\la \na R, \na f\ra=\Delta R+2|\Ric|^2=O(v^{-2}),
\ee
 we get 
\be
\begin{split}
-\frac{R-\Ric(\frac{\na f}{|\na f|}, \frac{\na f}{|\na f|})}{|\na f|^{2}}\la \na f, \na (vR)\ra&=O(v^{-1})\left(-|\na f|^2 R+v\la \na f, \na R\ra \right)\\
&=O(v^{-2}).
\end{split}
\ee 
For the Hessian term in the normal direction in \eqref{lap relation}, direct calculation and \eqref{Ric in na f} yield
\be
\begin{split}
\na^2(vR)(\na f, \na f)&=(vR)_{ij}f_if_j\\
&=vR_{ij}f_if_j+2v_iR_jf_if_j+Rv_{ij}f_if_j\\
&=v\na^2 R(\na f, \na f)-2|\na f|^2\la \na R, \na f\ra+R\Ric(\na f, \na f)\\
&=v\na^2 R(\na f, \na f)+O(v^{-2})+O(v^{-3})\\
&=v\na^2 R(\na f, \na f)+O(v^{-2}).
\end{split}
\ee
By \eqref{na R=Ric} $\na R=2\Ric(\na f)$, the second Bianchi identity and Shi's estimate \eqref{Shi est},
\be
\begin{split}
v\na^2 R(\na f, \na f)&=vR_{lkkl,ij}f_i f_j\\
&=2vR_{klli,kj}f_if_j\\
&=2v\left( R_{klli,k}f_i\right)_j f_j-2vR_{ki,k}f_{ij}f_j\\
&=2v\left[ \left(R_{klli}f_i\right)_k+|\Ric|^2\right]_j f_j+\frac{v|\na R|^2}{2}\\
&=v\la\na \Delta R, \na f\ra+2v\na_{\na f}|\Ric|^2+\frac{v|\na R|^2}{2}\\
&= O(v^{-\frac{3}{2}})+O(v^{-2})\\
&= O(v^{-\frac{3}{2}}).
\end{split}
\ee
Hence by \eqref{f lap v} and \eqref{f lap R}, there exists a constant $C_0>0$ such that
\be
\begin{split}
\Delta_{\Sigma_{\tau}}(vR)-\la \na f, \na (vR)\ra&\geq \Delta(vR)-\la \na f, \na (vR)\ra-O(v^{-\frac{3}{2}})\\
&=v\D R +R\D v+2\la \na R,\na v\ra-O(v^{-\frac{3}{2}})\\
&\geq-2v|\Ric|^2+R-C_0v^{-\frac{3}{2}}
\end{split}
\ee
near infinity.
\end{proof}

As an intermediate step toward the estimate on the curvature tensor $|\Rm|$, we use the maximum principle to establish the exponential decay of the norm of Ricci tensor $|\Ric|$.

\begin{lma}Under the same assumptions as Proposition \ref{level sets exp decay}, the Ricci curvature decays exponentially near infinity,
\be
|\Ric|\leq Ce^{-r}.
\ee
\end{lma}
\begin{proof}For all small $\ve>0$, we can find a large $a_1>1+a_0$ such that 
\be\label{choice of a1}
\frac{8C_0A_0^2}{\sqrt{a_1}}\leq \frac{1}{2} \text{  and  } \frac{2C_0}{\sqrt{a_1}}\leq \frac{\ve}{1000}.
\ee
where $a_0$, $C_0$ and $A_0$ are the constants in Lemma \ref{evolution of vR} and Proposition \ref{level sets exp decay} respectively.\\
\textbf{Claim:} For any $y$ in the set of  $\{v\geq a_1\}$, 
\be\label{para est}
v(y)R(y)\leq \ve.
\ee

We first assume the claim and prove the lemma. By \eqref{pot est vs r} (see also \cite[Lemma 2]{Chan-2019}), $\lim_{x\to\infty}r^{-1}v=1$ and hence from \eqref{para est} $\lim_{x\to\infty}rR=0$. Thanks to the assumption $|\Ric|\leq A_0 R$,
\be
\begin{split}
\Delta_f R&=-2|\Ric|^2\\
&\geq -2A_0^2R^2.
\end{split}
\ee
We may then apply Proposition $1$ and Lemma $3$ in \cite{Chan-2019} to conclude that 
\be
\begin{split}
|\Ric|&\leq A_0R\\
&\leq Ce^{-r}.
\end{split}
\ee
Now it remains to justify the claim, i.e. \eqref{para est}. Suppose the claim does not hold. Then there exists $y_0$  in the set of $\{v\geq a_1\}$ such that $v(y_0)R(y_0)> \ve$. By $\liminf_{x\to\infty}rR=0$ and Lemma \ref{local fast conv2}, we can find sequences of $p_i\to\infty$ with $\tau_i:=v(p_i)>v(y_0)$ and $\ve_i:=\sup_{\Sigma_{\tau_i}}vR\rightarrow 0$ as $i\to\infty$. Moreover for all large $i$, $\ve_i<\ve$. Therefore, we can always choose a $z_0$ $\in \{\tau_i\geq v\geq a_1\}$ with largest possible $v(z_0)$ such that $v(z_0)R(z_0)=\ve$. By the choices of $\ve_i$ and $z_0$, $v(z_0)<\tau_i$ and $vR<\ve$ on $\{\tau_i\geq v > v(z_0)\}$. We may invoke the parabolic maximum principle to conclude that at $z_0$
\be
\la \na f, \na (vR)\ra\geq 0;
\ee
\be
\Delta_{\Sigma_{v(z_0)}}(vR)\leq 0.
\ee
By $|\Ric|\leq A_0R$ and Lemma \ref{evolution of vR},
\be
0 \geq -2A_0^2v^2R^2+vR-C_0v^{-1/2}.
\ee
From the above inequality, we have two possible cases:\\
\textbf{Case 1:} 
\be
\begin{split}
\ve&=v(z_0)R(z_0)\\
&\geq \frac{1+\sqrt{1-\frac{8C_0A_0^2}{v^{1/2}}}}{4A_0^2}\\
&\geq \frac{3}{8A_0^2},
\end{split}
\ee
which is impossible for all small $\ve$. We also used $v(z_0)\geq a_1$ and \eqref{choice of a1} in the last inequality.\\
\textbf{Case 2:} \\

\be
\begin{split}
\ve&=v(z_0)R(z_0)\\
&\leq \frac{1-\sqrt{1-\frac{8C_0A_0^2}{v^{1/2}}}}{4A_0^2}\\
&=\frac{8C_0A_0^2v^{-1/2}}{(1+\sqrt{1-8C_0A_0^2v^{-1/2}})4A_0^2}\\
&\leq 2C_0v^{-1/2}\\
&\leq 2C_0a_1^{-1/2}.
\end{split}
\ee
It follows from \eqref{choice of a1} that
\be
\ve\leq  2C_0a_1^{-1/2}\leq \frac{\ve}{1000}. 
\ee
We again arrive at a contradiction. This justifies our claim and thus completes the proof.

\end{proof}

\begin{proof}[\textbf{Proof of Proposition \ref{level sets exp decay}:}] Recall that our goal is to show $|\Rm|\leq Ce^{-r}$. Deruelle \cite[Lemma 2.8]{Deruelle-2017} proved a local derivative estimate for tensor $T$ on an expanding gradient soliton satisfying an elliptic equation of the form
\be
\D T=-\lambda T+ \Rm\ast T.
\ee
As pointed out by him in \cite{Deruelle-2017}, the same argument also works for steady gradient Ricci soliton. Hence we apply his result with $T=\Ric$ and $\lambda=0$ to get
\be\label{Hess est for Ric}
|\na^2 \Ric|\leq Ce^{-r}\leq C'e^{f}.
\ee
Let $\psi_s$ be the flow of the vector field $-\frac{\na f}{|\na f|^2}$ with $\psi_0=\text{  id}$. For all $q$ near infinity, $\psi_s(q)$ is well defined for all $s\geq 0$ and 
\be
f(\psi_s(q))=f(q)-s.
\ee
Hence $\psi_s(q)\rightarrow \infty$ as $s\to\infty$. It follows from the Ricci identity and the soliton equation (see also \eqref{Ric identity}) that
$$R_{kl,i}-R_{ki,l}=R_{ilkj}f_j.$$
By the second Bianchi identity and \eqref{Hess est for Ric}, we differentiate the quantity $|\Rm|^2$ along the flow 
\be
\begin{split}
\frac{\p}{\p s}|\Rm|^2(\psi_s(q))&=-\frac{\la\na |\Rm|^2, \na f\ra}{|\na f|^2}\\
&=-2|\na f|^{-2}R_{ijkl}R_{ijkl, \a}f_{\a}\\
&=-2|\na f|^{-2}R_{ijkl}\big[-(R_{ijl\a}f_{\a})_k+(R_{ijk\a}f_{\a})_l-R_{ijl\a}R_{\a k}+R_{ijk\a}R_{\a l}\big]\\
&\geq -Ce^{f(q)-s}|\Rm|-Ce^{f(q)-s}|\Rm|^2\\
&\geq -2Ce^{f(q)-s}-2Ce^{f(q)-s}|\Rm|^2.
\end{split}
\ee
We may integrate the above inequality and get
\be
\ln\left(\frac{1+|\Rm|^2(\psi_s(q))}{1+|\Rm|^2(q)}\right)\geq 2Ce^{f(q)-s}-2Ce^{f(q)}.
\ee
Letting $s \to\infty$, together with the fact that $\lim_{x\to\infty}|\Rm|=0$, one has 
\be\label{exp in f}
\begin{split}
|\Rm|^2(q)&\leq e^{2Ce^{f(q)}}-1\\
&\leq 2Ce^{2Ce^{f(q)}}e^{f(q)}\\
&\leq C^{'''}e^{f(q)}.
\end{split}
\ee
We also used $f\leq -10$ on $M$ in the last inequality. Thank to \eqref{pot est vs r} and \eqref{exp in f}, we see that $\lim_{x\to\infty}r|\Rm|=0$. We then apply \cite[Theorem 2]{Chan-2019} (see also \cite{MunteanuSungWang-2017}) and conclude that
\be
|\Rm|\leq Ce^{-r}.
\ee

\end{proof}

\begin{proof}[\textbf{Proof of Theorem \ref{4 dim dichotomy}:}] By our assumption $R \to 0$ as $x\to \infty$. It then follows from  \cite{Chan-2019} that
\be\label{4 dim Rm by R in proof}
|\Ric|\leq c_1 |\Rm| \leq A_0 R \text{  on } M.
\ee

\begin{itemize}
    \item If $\liminf_{x\to\infty}rR>0$, then $R \sim r^{-1}$ near infinity and estimate \eqref{4 dim linear decay rate} on $\Rm$ is a consequence of \eqref{4 dim Rm by R in proof}.
    \item If $\liminf_{x\to\infty}rR=0$, we apply Proposition \ref{level sets exp decay} to see that
\be
|\Rm|\leq Ce^{-r}.
\ee
By the lower bound for $R$ established by Chow-Lu-Yang \cite{ChowLuYang-2011} (see also \cite{MunteanuSungWang-2017}), $|\Rm|$ is bounded below by a constant multiple of $e^{-r}$.
\end{itemize}

\end{proof}
\begin{proof}[\textbf{Proof of Theorem \ref{n dim dichotomy}:}] Since $\Ric\geq 0$ outside a compact subset, there exists a large positive constant $A_0$ such that
\be
|\Ric|\leq A_0R.
\ee 
Using $|\Rm|\to 0$ as $x\to \infty$ and an estimate on the potential function $f$ in \cite[Proposition 1]{Chan-2019.2} (see also \cite{ChowDengMa-2020}), we have $\lim_{x\to\infty}f=-\infty$. Theorem \ref{n dim dichotomy} then follows similarly from Proposition \ref{level sets exp decay} as in the proof of Theorem \ref{4 dim dichotomy}.
\end{proof}

\bigskip

\section{Proof of Theorem \ref{fast decay soliton}}

We first investigate the topology of the level sets of $f$ and show that they are diffeomorphic to a quotient of torus at infinity if the curvature decays sufficiently fast.

\begin{prop}\label{level sets are e flat}
Let $(M^n,g, f)$ be a complete noncompact non Ricci flat steady gradient Ricci soliton of real dimension $n$ with proper potential function. Further suppose that the following conditions are satisfied:
\begin{itemize} 
\item $|\Ric|\leq A_0R$ on $M$ for some constant $A_0>0$;
\item $|\Rm|\leq A_1(r+1)^{-1}$ on $M$ for some constant $A_1>0$;
\item $\liminf_{r\to\infty}rR=0$.
\end{itemize}
Then for all sufficiently large $\tau$, the level sets $\Sigma_{\tau}:=\{x: -f(x)=\tau\}$ are diffeomorphic to a finite quotient of torus. 
\end{prop}
\begin{proof}
By Proposition \ref{level sets exp decay} and a volume estimate by Munteanu-Sesum \cite{MunteanuSesum-2013}, we see that $|\Rm|$ is integrable, i.e. in $L^1(M)$. Thanks to the properness of $f$ and a result of Munteanu-Wang \cite{MunteanuWang-2011}, $M$ is connected at infinity and hence the level sets $\{-f=\tau\}$ are connected smooth compact hypersurfaces in $M$ for all large $\tau$. Moreover, there exists a large $\tau_0$ such that $\na f\neq 0$ on $\{-f\geq \tau_0\}$ and  $\Sigma_{\tau}=\{-f=\tau\}$ are diffeomorphic to each other for all $\tau\geq \tau_0$. The same argument used by Deruelle \cite[Proposition 2.3]{Deruelle-2012} shows that the level sets are diffeomorphic to a compact flat manifold. Indeed. if $\psi_s$ is the flow of $-\frac{\na f}{|\na f|^2}$ with $\psi_0=\text{  id}$. Then 
$\psi_s: \Sigma_{\tau_0}\longrightarrow \Sigma_{\tau_0+s}$ are diffeomorphisms for all $s\geq 0$ and when restricted on the tangent space of the level set $T\Sigma_{\tau_0}$, the pull back metric satisfies
\be\label{evolution of g in exp}
\begin{split}
\frac{\p}{\p s}\psi_s^*g&=-\psi_s^*\mathcal{L}_{\frac{\na f}{|\na f|^2}}g\\
&=\psi_s^*\left(\frac{2\Ric}{|\na f|^2}\right)\\
&= O\left(e^{-\tau_0-s}\right)\psi_s^*g.
\end{split}
\ee
Proposition \ref{level sets exp decay} was used in the last equation. Hence $\psi_s^*g$ are uniformly equivalent to each other for all $s\geq 0$ and we can find a constant $C>0$ such that for all $\tau\geq \tau_0$
\be\label{diam relation}
\begin{split}
\text{Diam}(\Sigma_{\tau})&\leq C \text{  Diam}(\Sigma_{\tau_0});\\
\text{Vol}(\Sigma_{\tau_0})&\leq C \text{  Vol}(\Sigma_{\tau}).
\end{split}
\ee
Similar argument by Deruelle \cite[Lemma 2.5]{Deruelle-2012} shows that for all nonnegative integer $k$, there exists $C_k$ independent of $\tau$ such that
\be
|\na_{\Sigma_{\tau}}^k\Rm_{\Sigma_{\tau}}|\leq C_k,
\ee
where $\Rm_{\Sigma_{\tau}}$ and $\na_{\Sigma_{\tau}}$ denote the curvature tensor and the Riemannian connection of $\Sigma_{\tau}$ with respect to the induced metric by $g$. It follows from $\lim_{\tau \to \infty}\sup_{\Sigma_{\tau}} |\Rm_{ \Sigma_{\tau}}|=0$ and the Hamilton Compactness Theorem \cite{Hamilton-1995.2} that for any sequence $\tau_i\to\infty$, by passing to subsequence if necessary, $\Sigma_{\tau_i}$ converges smoothly to a compact flat manifold as $i \to \infty$. Hence, for large $\tau$, $\Sigma_{\tau}$ is diffeomorphic to a finite quotient of the torus $\mathbb{T}^{n-1}$ by the Bieberbach's Theorem (see \cite{Deruelle-2012} and reference therein).
\end{proof}

With the topological restriction on the level sets in Proposition \ref{level sets are e flat}, we will prove that the level sets $\Sigma_{\tau}$ are flat with respect to the induced metric for all large $\tau$. The key observation is that torus does not admit any nontrivial metric with nonnegative scalar curvature. This approach was used by Deruelle \cite{Deruelle-2012} to study steady soliton with integrable curvature.

\begin{lma}\label{flat level sets}Under the assumptions in Theorem \ref{fast decay soliton}, for all sufficiently large $\tau$, the level sets $\Sigma_{\tau}:=\{x: -f(x)=\tau\}$ with induced metrics $g|_{\Sigma_{\tau}}$ from $M$ are flat.
\end{lma}
\begin{proof} By Remark \ref{Rm by R in n dim rmk}, $|\Rm|\leq cR$ for some positive constant $c>0$ (when $m=2$, the estimate follows from \cite{Chan-2019}, where $m$ is the complex dimension). It can be seen from \cite[Proposition 1]{Chan-2019.2} (see also \cite{ChowDengMa-2020}) that $f$ is proper. We may then invoke Proposition \ref{level sets are e flat} to see that the level sets $\Sigma_{\tau}:=\{x: -f(x)=\tau\}$ are diffeomorphic to a quotient of torus.

We are going to show that they have nonnegative scalar curvature. Using the properness of $f$, we may choose larger $\tau$ such that $\Ric\geq 0$ on $\{x: -f(x)\geq \tau\}$. The second fundamental form of $\Sigma_{\tau}$ (w.r.t. the normal $-\frac{\na f}{|\na f|}$) is given by $-\frac{\na^2 f}{|\na f|}=\frac{\Ric}{|\na f|}\geq 0$. Let $0\leq \mu_1\leq \mu_2\leq \cdots \leq \mu_{2m-1}$ be the eigenvalues of the second fundamental form of $\Sigma_{\tau}$. Then we must have 
\be\label{H and A}
\begin{split}
H_{\Sigma_{\tau}}^2-|A_{\Sigma_{\tau}}|^2&=\left(\sum_{k=1}^{2m-1}\mu_k\right)^2-\sum_{k=1}^{2m-1}\mu_k^2\\
&= \sum_{1\leq k<l}^{2m-1}2\mu_k\mu_l\\
&\geq 0,
\end{split}
\ee
where $H_{\Sigma_{\tau}}$ and $A_{\Sigma_{\tau}}$ refer to the mean curvature and the second fundamental form of $\Sigma_{\tau}$ respectively. By the Gauss equation,  the intrinsic scalar curvature of  $\Sigma_{\tau}$, denoted by $R_{\Sigma_{\tau}}$ satisfies 
\be\label{intrinsic scalar lower bdd}
\begin{split}
R_{\Sigma_{\tau}}&= R-\frac{2\Ric(\na f, \na f)}{|\na f|^2}+H_{\Sigma_{\tau}}^2-|A_{\Sigma_{\tau}}|^2\\
                            &\geq R-\frac{2\Ric(\na f, \na f)}{|\na f|^2}\\
                            &= R-\frac{\Ric(\na f, \na f)}{|\na f|^2}-\frac{\Ric(J\na f, J\na f)}{|\na f|^2}\\
&\geq 0,
\end{split}
\ee
we also used the K\"{a}hlerity of ambient metric $g$ in the last equality.
Thus, the metric on $\Sigma_{\tau}$ is of nonnegative scalar curvature and hence is flat since there exists no nonflat metric on torus with nonnegative scalar curvature (See \cite{GL-spin, SY}).
\end{proof}

Now,
we are in a position to prove Theorem \ref{fast decay soliton}.

\begin{proof}[\textbf{Proof of Theorem \ref{fast decay soliton}:}]
By Lemma \ref{flat level sets}, we have $\Sigma_{\tau}$ are flat for all large $\tau$. Hence equalities hold in \eqref{intrinsic scalar lower bdd} and we have the following identities outside a compact subset of $M$
\begin{eqnarray}
\label{non 0 eigenvector}R|\na f|^2&=&2\Ric(\na f, \na f)\\
\label{Ric kernel}R^2&=&2|\Ric|^2
\end{eqnarray}
Using a result of \cite{Kotschwar-2013} (see also \cite{DeTurckKazdan-1981,Ivey-1996}): Ricci solitons are real-analytic. Hence, the soliton metric $g_{ij}$ is real analytic in its geodesic normal coordinates. Moreover $f$ satisfies
\be
\Delta f=-R 
\ee
and by the elliptic regularity theory (see \cite[p.110]{GilbargTrudinger-2001} and ref. therein), $f$ is also real analytic in the geodesic normal coordinates. 

Thanks to the analytic continuation, we see that \eqref{non 0 eigenvector} and \eqref{Ric kernel} indeed hold globally on $M$. Since the set of critical points $\{|\na f|=0\}$ is nowhere dense in $M$, it follows from \eqref{non 0 eigenvector} and \eqref{Ric kernel} that $\Ric\geq 0$ on $M$ and the kernel of $\Ric$ is a smooth subbundle of the tangent bundle of $M$ with real rank $2m-2$. By the K\"{a}hlerity of $g$, $\Ric$ only has two distinct eigenvalues, namely $0$ with real multiplicity $2m-2$ and $\frac{R}{2}$ with real multiplicity $2$. Moreover, $\na f$ and $J\na f$ span the eigenspace of the eigenvalue $\frac{R}{2}$ wherever $\na f \neq 0$. 

To prove the splitting of $M$, we proceed as in \cite{Chan-2019.2} to show that the kernel of $\Ric$ is invariant under parallel translation. Let $E$ be the kernel of $\Ric$, it is a smooth subbundle of tangent bundle of real rank $2m-2$ (see also \cite{MunteanuSungWang-2017}). Suppose at $p$, $\na f\neq 0$, by the orthogonal decomposition, the tangent space at $p$ can be splitted into $T_pM$ $=$ $E_p$ $\oplus_{\perp}$ $\text{span}\{\na f, J\na f\}$. Let $X$ be a smooth section of $E$ defined locally near $p$ and $Y$ be any smooth vector field defined near $p$, then $JX$ is also a smooth section of $E$. At $p$
\begin{eqnarray*}
\la\na_Y X,\na f\ra &=& Y\la X, \na f\ra- \la X, \na_Y\na f\ra\\
&=& \Ric(X, Y)\\
&=& 0.
\end{eqnarray*}
Similarly, $\na_Y JX \perp \na f$, thus $\na_Y X(p)$ is in $E_p$. If $\na f=0$ at $p$, by the real analyticity of $g$ (see \cite{DeTurckKazdan-1981,Ivey-1996, Kotschwar-2013}), $\{\na f=0\}=\{R=1\}$ has no interior point in $M$, we may find a sequence $p_k \to p$ as $k\to\infty$ with $\na f(p_k)\neq 0$,
$$\Ric(\na_Y X)(p)=\lim_{k\to \infty}\Ric(\na_Y X)(p_k)=0.$$
Hence, we conclude that $E$ is invariant under parallel translation. 

By the De Rham Splitting Theorem \cite{KobayashiNomizu-1963, KobayashiNomizu-1969}, the universal covering space of $M$ splits isometrically and holomorphically as $M_1\times M_2$, where $M_1$, $M_2$ are complex $m-1$ and $1$ dimensional K\"{a}hler manifolds respectively. Moreover, the tangent bundle of $M_1$ is given by the kernel of $\Ric$ and the tangent bundle of $M_2$ coincides with the nonzero eigenspace of $\Ric$. From this we conclude that $M_1$ is Ricci flat and $M_2$ is nonflat. $M$ also induces a steady K\"{a}hler gradient soliton structure on $M_2$ and by the classification of complex $1$ dimensional nontrivial complete steady K\"{a}hler gradient Ricci soliton \cite{ChowLuNi-2006}, $M_2$ is holomorphically isometric to the Cigar soliton $\Sigma$. When $m=2$, $M_1$ is flat. If $m\geq 3$, since $|\Rm|\leq cR$ (see \eqref{Rm by R in n dim}) also holds on $M_1\times \Sigma$ and $M_1$ is Ricci flat, for any $(a,b)$ $\in  M_1\times \Sigma$,
\begin{eqnarray*}
|\Rm_{M_1}|(a)&\leq& |\Rm_{M_1\times \Sigma}|(a, b)\\
&\leq& cR_{\Sigma}(b) \to 0 \text{  as } b\to \infty,
\end{eqnarray*}
we used the fact that the curvature of $\Sigma$ decays at infinity. Hence $M_1$ is also flat for $m\geq 3$. The exponential curvature decay is a consequence of the conditions $\liminf_{x\to\infty}rR=0$, Theorems \ref{4 dim dichotomy} and \ref{n dim dichotomy}. This completes the proof of the theorem.


\end{proof}

\section{Proof of Theorem \ref{fast decay Rm soliton}}

We begin with an elementary lemma on the regularity of the component functions of a parallel vector field on real analytic manifolds.
\begin{lma}\label{analytic components}
Let $(M^n,g)$ be a complete Riemannian manifold such that in any geodesic normal coordinates, the corresponding metric coefficients $g_{ij}$ are real analytic functions. Suppose that $\gamma: [a, b] \longrightarrow M$ is a geodesic and $V(t)$ is a parallel vector field along $\gamma(t)$. For any $t_0$ $\in (a, b)$ and $\{x_i\}$ a geodesic normal coordinate centered at $\gamma(t_0)$, the component functions $V^i(t)$ of $V$ are real analytic functions in $t$.
\end{lma}
\begin{proof} 
Let $t_0$ and $\{x_i\}$ be the number and geodesic coordinate as in the statement of the lemma. In the local coordinate, $\gamma(t)=(t-t_0)a$ for some $a$ $\in \R^n$, $V(t)=V^i (t) \frac{\p}{\p x_i}( \gamma(t) )$ and $\frac{D V}{dt}=0$ can be rewritten as
\be
\frac{d V^i(t)}{dt}+a^jV^k(t)\Gamma^i_{jk}((t-t_0)a)=0,
\ee
where $\Gamma^i_{jk}$ are the Christoffel symbols. Since $g_{ij}(x)$ are real analytic, $\Gamma^i_{jk}(x)$ are also real analytic in $x$ and hence $ \Gamma^i_{jk}((t-t_0)a)$ are real analytic in $t$. By \cite [Theorem 1.3]{Chicone-2006}, $V^i(t)$, as the solution to the above ODE, is real analytic on $(t_0-\delta, t_0+\delta)$ for some $\delta>0$. 
\end{proof}
With the above preparations, we shall prove Theorem \ref{fast decay Rm soliton}.
\begin{proof}[\textbf{Proof of Theorem \ref{fast decay Rm soliton}:}] Since the sectional curvature of $M$ is nonnegative near infinity and $R\to 0$ as $x\to \infty$, $f$ is proper (\cite{Chan-2019.2, ChowDengMa-2020}) and $|\Ric|\leq A_0R$ on $M$ for some positive constant $A_0$. Using Proposition \ref{level sets are e flat}, the level sets $\Sigma_{\tau}=\{-f=\tau\}$ are diffeomorphic to a quotient of torus $\mathbb{T}^{n-1}$ for all sufficiently large $\tau$. If $M$ has nonnegative sectional curvature at $p$, then for all unit tangent vector $v$ $\in T_pM$, 
\be\label{Ric bdd above by R in all directions}
R\geq 2\Ric(v,v).
\ee
Therefore, we can apply the Gauss equation as in \eqref{intrinsic scalar lower bdd} to see that the intrinsic scalar curvature $R_{\Sigma_{\tau}}$ is nonnegative for all large $\tau$. Since it is well-known that on $\mathbb{T}^{n-1}, n \geq 2$, any metric with nonnegative scalar curvature is flat (See \cite{GL-spin, SY}). Hence, the induced metric on $\Sigma_{\tau}$ is flat and the equality in \eqref{H and A} holds and
\be\label{R=Rnn}
R= \frac{ 2\Ric(\na f, \na f)}{ |\na f|^2}.
\ee
It is thanks to the above equation and \eqref{Ric bdd above by R in all directions} that $\na f$ is an eigenvector of $\Ric$ and $\Ric(\na f, v)=0$ for all $v$ $\in T\Sigma_{\tau}$. Hence
\be\label{R2=2Ric2}
\begin{split}
2|\Ric|^2&=2R_{aa}^2+ 2|\na f|^2 |A_{\Sigma_{\tau}}|^2\\
&=2R_{aa}^2+ 2|\na f|^2 H_{\Sigma_{\tau}}^2\\
&=2R_{aa}^2+2(R-R_{aa})^2\\
&= R^2,
\end{split}
\ee
where $a=\frac{\na f}{|\na f|}$. We also used \eqref{R=Rnn} in the last equality. It can be seen from \eqref{R2=2Ric2} and an argument by Munteanu-Sung-Wang \cite[Proposition 5.4]{MunteanuSungWang-2017} that wherever $M$ has nonnegative sectional curvature, $\Ric$ has two distinct eigenvalues, $0$ with multiplicity $n-2$ and $R/2$ with multiplicity $2$.

Let $K$ be a compact set such that $M$ has nonnegative sectional curvature and $\na f\neq 0$ on $M\setminus K$. By the strong maximum principle \cite[Theorem 12.50]{Chowetal-2008}, the kernel of the Ricci tensor $\Ric$ is invariant under parallel translation on $M\setminus K$. Due to the De Rham Splitting Theorem \cite[Theorem 10.3.1]{Petersen-2016}, for all $p$ $\in M\setminus K$, there are open neighborhood $U$ of $p$ in $M\setminus K$, manifolds $(U_1^{n-2}, g_1)$  and $(U_2^{2}, g_2)$ such that the following isometric splitting is true 
\be\label{isom split}
\begin{split}
(U, g|_U) &\cong (U_1\times U_2, g_1+g_2);\\
TU_1&=\text{  null}(\Ric);\\
TU_2&=\text{  null}(\Ric)^{\perp},\\
\end{split}
\ee
where $\text{null}(\Ric)$ denotes the nullspace of the Ricci curvature. It can be seen from the splitting that $(U_1^{n-2}, g_1)$ is flat. 

We are going to show that $M$ has nonnegative sectional curvature everywhere. Once it is established, Theorem \ref{fast decay Rm soliton} will be a consequence of the result by Deng-Zhu \cite{DengZhu-2018} (see also \cite{MunteanuSungWang-2017}). Fix $p$ $\in M\setminus K$ and open set $U$ which splits isometrically as in \eqref{isom split}, for any $q$ $\in K$, let $\gamma: [0, d]\longrightarrow M$ be a normalized geodesic joining $p$ to $q$, i.e. $\gamma(0)=p$ and $\gamma(d)=q$. Let $\{e_i\}_{i=1}^{n-2}$ and $\{\mu_i\}_{i=1}^2$ be orthonormal bases for $\text{null}(\Ric)$ and $\text{null}(\Ric)^{\perp}$ at $p$ respectively.  Their parallel translations along $\gamma$ are denoted by $\{e_i(t)\}_{i=1}^{n-2}$ and $\{\mu_i(t)\}_{i=1}^2$. For any parallel vector fields $A(t)$ and $B(t)$ along $\gamma$, we consider the following identities
\be\label{id 1}
R(e_i(t), e_j(t), e_k(t), e_l(t))(\gamma(t))=0 \text{  for } i, j, k, l=1,\cdots, n-2;
\ee
\be\label{id 2}
R(e_i(t), \mu_j(t), A(t), B(t))(\gamma(t))=0 \text{  for } i=1,\cdots, n-2, j=1, 2;
\ee
\be\label{id 2.1}
R(e_i(t), A(t), \mu_j(t), B(t))(\gamma(t))=0 \text{  for } i=1,\cdots, n-2, j=1, 2;
\ee
\be\label{id 3}
2R(\mu_1(t), \mu_2(t), \mu_2(t), \mu_1(t))(\gamma(t))-R(\gamma(t))=0;
\ee
\be
L:=\sup\{s\in [0, d] : \eqref{id 1},  \eqref{id 2}, \eqref{id 2.1} \text{  and }  \eqref{id 3} \text{  hold for all } t\,\in[0, s]\}.
\ee
By the isometric splitting of $U$ in \eqref{isom split} and the invariance of $\text{null}(\Ric)$ and $\text{null}(\Ric)^{\perp}$ under parallel translation on $U$, we see that $L>0$. 

We claim that $L=d$. Suppose on the contrary, it follows from a result of Kotschwar \cite{Kotschwar-2013} that the metric coefficients $g_{ij}$ of a Ricci soliton are real analytic in the geodesic normal coordinates (see also \cite{DeTurckKazdan-1981,Ivey-1996}). By Lemma \ref{analytic components}, the component functions of $\{e_i(t)\}_{i=1}^{n-2}$, $\{\mu_i(t)\}_{i=1}^2$, $A(t)$ and $B(t)$ are real analytic in $t$ near $L$ in the geodesic normal coordinate of $(M, g)$ centered at $\gamma(L)$.
Hence, the L.H.S. of \eqref{id 1}, \eqref{id 2}, \eqref{id 2.1} and \eqref{id 3} are real analytic functions in $t$ on $(L-\delta, L+\delta)$ and vanish on $(L-\delta, L]$ for some $\delta>0$. Thanks to the analytic continuation, the L.H.S. of \eqref{id 1}, \eqref{id 2}, \eqref{id 2.1} and \eqref{id 3} are identically zero on $(L-\delta, L+\delta)$ and thus $L\geq L+\delta$, which is absurd. This justifies our claim. 

Finally. For any $A$, $B$ $\in T_qM$, we may write
\be
\begin{split}
A&=\sum_{i=1}^{n-2}A'_ie_i(d)+\sum_{\a=1}^{2}A''_{\a}\mu_{\a}(d);\\
B&=\sum_{i=1}^{n-2}B'_ie_i(d)+\sum_{\a=1}^{2}B''_{\a}\mu_{\a}(d),
\end{split}
\ee
where $A'_i$, $A''_{\a}$, $B'_i$ and $B''_{\a}$ are some constants independent on $t$. Hence by  \eqref{id 1}, \eqref{id 2}, \eqref{id 2.1}, \eqref{id 3} and $L=d$,
\begin{eqnarray*}
R(A, B, B, A)(q)
&=&A'_iB'_jB'_kA'_lR(e_i(d), e_j(d), e_k(d), e_l(d))\\
&&+A''_{\a}B''_{\beta}B''_{\gamma}A''_{\delta}R(\mu_{\a}(d),\mu_{\b}(d), \mu_{\gamma}(d), \mu_{\delta}(d))\\
&=&\frac{1}{2}\left(A''_{2}B''_{1}-A''_{1}B''_{2}\right)^2 R(\gamma(d))\\
&\geq&0.
\end{eqnarray*}
Thus $M$ has nonnegative sectional curvature everywhere. Result then follows from the strong maximum principle argument as in \cite{DengZhu-2018, MunteanuSungWang-2017}.

\end{proof}

\section{Gromov-Hausdorff Limit at infinity}
In this section, we prove the following proposition which implies Corollary \ref{GH limit}. Moreover, $\Ric\geq 0$ near infinity is not needed in the proposition.
\begin{prop}\label{GH limit2} Let $(M^n, g, f)$ be a complete non Ricci flat steady gradient Ricci soliton with dimension $n\geq 4$ and proper potential function $f$.
\begin{enumerate}
   \item If \eqref{4 dim linear decay rate} or \eqref{lin decay in n dim} holds, then for any $p_i\to \infty$ in $M$, after passing to a subsequence, $(M, d_{R(p_i)g}, p_i)$ converges in pointed Gromov-Hausdorff sense to a cylinder $(\R\times Y, \sqrt{d_e^2+d_Y^2}, p_{\infty})$, where $d_e$ is the flat metric on $\R$,  $(Y,d_Y)$ denotes a compact Alexandrov space and $\sqrt{d_e^2+d_Y^2}$ indicates the product metric.
   \item If instead \eqref{4 dim exp decay rate} or \eqref{exp in n dim} is true, then for any $p_i\to \infty$ in $M$, $(M, d_{R(p_i)g}, p_i)$ converges in pointed Gromov-Hausdorff sense (without passing to subsequence) to the ray $([0,\infty), d_e, 0)$,  where $d_e$ is the flat metric restricted on $[0, \infty)$. In this case, $(M,g)$ is smoothly asymptotic to the cylinder $\R\times\left(\mathbb{T}^{n-1}\big/\sim\right)$ with flat product metric at exponential rate, where $\mathbb{T}^{n-1}\big/\sim$ is diffeomorphic to the quotient of torus in Corollary \ref{level set dich}$(a)$.
   \end{enumerate}
\end{prop}

\begin{proof}[Proof of Proposition \ref{GH limit2}:]
By the properness of $f$ \eqref{pot est vs r}, there is a large $\tau_0$ such that on $\{-f\geq \tau_0\}$
\be\label{lev vs r}
2^{-1}r\leq -f\leq 2r.
\ee
Let $R_i=R(p_i)$, $\tau_i=-f(p_i)\to\infty$, $h_i=R(p_i)g$ and $\tilde{h}_i=h_i|_{\Sigma_{\tau_i}}$. It follows from Lemma \ref{level set diam est} that the intrinsic diameter of $\Sigma_{\tau_i}=\{-f=\tau_i\}$ with respect to the scaled metric $\tilde{h}_i=h_i|_{\Sigma_{\tau_i}}$ is uniformly bounded from above
\be
\begin{split}
    \text{diam  }(\Sigma_{\tau_i}, \tilde{h}_i)&\leq C\sqrt{\tau_iR_i} \\
    &\leq C'\sqrt{r(p_i)^{-1}\tau_i}\\
    &\leq C'\sqrt{2}.
\end{split}
\ee
We will separate the argument into two cases, namely linear and exponential curvature decays.\\
\\
\textbf{Case $(a)$:} Linear curvature decay\\
\\
We shall apply the level set method by Deng-Zhu \cite{DengZhu-2018, DengZhu-2019, DengZhu-2020.2} to show the convergence. By the Gauss equation, \eqref{4 dim linear decay rate} or \eqref{lin decay in n dim} and \eqref{lev vs r}, we see that $c^{-1}\tau_i^{-1}\leq R_i\leq c\tau_i$ and  
\begin{eqnarray*}
|\Rm_{\Sigma_{\tau_i}}(\tilde{h}_i)|&\leq& CR_i^{-1}\left(|\Rm(g)|+\frac{|\Ric(g)|^2}{|\na_g f|^2}\right)\\
&\leq& C'\tau_i\left(|\Rm(g)|+\frac{|\Ric(g)|^2}{|\na_g f|^2}\right)\\
&\leq& C''\tau_i\left(\tau_i^{-1}+\tau_i^{-2}\right)\\
&\leq& C'''.
\end{eqnarray*}
By Gromov Compactness theorem \cite[Theorem 10.7.2]{BuragoBuragoIvanov-2001} (see also \cite{Petersen-2016}), 
after passing to a subsequence,
\be\label{level set GH conv}
(\Sigma_{\tau_i}, d_{\tilde{h}_i}, p_i) \longrightarrow (Y, d_Y, p_{\infty})
\ee
converges in pointed Gromov-Hausdorff topology as $i\to \infty$, where $(Y, d_Y, p_{\infty})$ is a compact Alexandrov space. 

Next we consider a type of sets introduced by Deng-Zhu \cite{DengZhu-2018, DengZhu-2019, DengZhu-2020.2}, for any $p$ $\in M$ and $k>0$
\be\label{mpk}
M_{p, k}:=\left\{y:\, |f(y)-f(p)|\leq \frac{k}{\sqrt{R(p)}}\right\}.
\ee
Moreover, they \cite[Lemma 3.1]{DengZhu-2019} showed that for all $k>0$, there is a large $I$ such that for all $i\geq I$
\be\label{ball vs mpk}
B_{h_i}(p_i, k)\subseteq M_{p_i, k}.
\ee
Let $\psi_{s}$ be the flow of the vector field $-\frac{\na_g f}{|\na_g f|_g^2}$ with $\psi_0=\text{  id }$. Then for all large $i$, $\psi_s: \Sigma_{\tau_i}\longrightarrow \Sigma_{\tau_i+s}$ are diffeomorphisms for $s$ $\in [-\frac{k}{\sqrt{R_i}}, \frac{k}{\sqrt{R_i}}]$. Hence we can define the following diffeomorphism $\Gamma_i: [-k, k]\times \Sigma_{\tau_i} \longrightarrow M_{p_i,k}$, where 
\be
\Gamma_i(s, q):=\psi_{\frac{s}{\sqrt{R_i}}}(q).
\ee
We compare the pull back metrics as in \cite[Lemma 4.2]{DengZhu-2020.2}. When restricting on $T\Sigma_{\tau_i}$,
\be
\begin{split}
  \frac{\p}{\p s} \psi_{\frac{s}{\sqrt{R_i}}}^*h_i &=\sqrt{R_i}\psi_{\frac{s}{\sqrt{R_i}}}^*\left(\frac{2\Ric}{|\na_g f|^2}\right)\\
   &\leq C\sqrt{R_i}\psi_{sR_i^{-1/2}}^*\left(\frac{g}{r}\right)\\
    &\leq \frac{2C\psi_{sR_i^{-1/2}}^*h_i}{s+\tau_i\sqrt{R_i}}.
    \end{split}
\ee
Similarly,
\be
\frac{\p}{\p s} \psi_{\frac{s}{\sqrt{R_i}}}^*h_i \geq -\frac{2C\psi_{sR_i^{-1/2}}^*h_i}{s+\tau_i\sqrt{R_i}}.
\ee
By integrating the above differential inequalities, we have for all large $i$ and $s\,\in [-k,k]$
\be
\left(1-\frac{2k}{\tau_i\sqrt{R_i}}\right)^{2C}h_i\leq \psi_{\frac{s}{\sqrt{R_i}}}^*h_i \leq \left(1+\frac{2k}{\tau_i\sqrt{R_i}}\right)^{2C}h_i \quad\text{on  } T\Sigma_{\tau_i}.
\ee
Since $R\sim (-f)^{-1}$, we also have 
\be
1\leq \left|\frac{\p}{\p s}\psi_{sR_i^{-1/2}}\right|_{h_i}^2=\frac{1}{|\na_g f|^2_g}=(1-R)^{-1}\leq (1-\frac{2C'}{\tau_i-ck\sqrt{\tau_i}})^{-1}.
\ee
Hence by $R_i\sim \tau_i^{-1}$, we conclude that on $[-k, k]\times \Sigma_{\tau_i}$, for all large $i$ (fixing $k>0$)
\be\label{unif equiv for lin conv}
(1-o(1))\left(ds^2+\tilde{h}_i\right)\leq \Gamma^*_i h_i \leq (1+o(1))\left(ds^2+\tilde{h}_i\right),
\ee
where $\tilde{h}_i=h_i|_{\Sigma_{\tau_i}}$. In view of \eqref{level set GH conv}, for any $\ve>0$, we can consider a sequence of Gromov Hausdorff approximations $F_i: (\Sigma_{\tau_i}, d_{\tilde{h}_i}, p_i) \longrightarrow (Y, d_Y, p_{\infty})$. Using \eqref{level set GH conv}, \eqref{ball vs mpk} and \eqref{unif equiv for lin conv}, one may check that $(\text{id  }, F_i)\circ \Gamma_i^{-1}$ is an  $\ve$ isometry from $M_{p_i, k}$ to $[-k, k]\times Y$ for all large $i$. This implies the pointed Gromov-Hausdorff convergence to the product space $\left(\R\times Y, \sqrt{d_e^2+d_
  Y^2}, (0, p_{\infty})\right)$ and finishes the proof for Proposition \ref{GH limit2} in Case $(a)$.
  \begin{remark} It can be seen from the Gauss equation, \eqref{Ric identity} and Shi's estimate \eqref{Shi est} that $(\Sigma_{\tau_i}, g_{\tilde{h}_i})$ has uniformly positive scalar curvature,
\be
\begin{split}
R_{\tilde{h}_i}&\geq R_i^{-1}R_g-2R_i^{-1}\Ric_g\left(\frac{\na f}{|\na f|}, \frac{\na f}{|\na f|}\right)-cR_i^{-1}|\Ric(g)|^2\\
&\geq \quad c-c'\tau_i^{-1}.
\end{split}
\ee
  
  \end{remark}
  \noindent
\textbf{Case $(b)$:} Exponential curvature decay\\
\\
Again by Lemma \ref{level set diam est} and $R\leq Ce^{-r}\leq C'e^f$,
\be\label{diam control in GH exp}
\begin{split}
    \text{diam  }(\Sigma_{\tau_i}, \tilde{h}_i)&\leq C\sqrt{\tau_iR_i} \\
    &\leq C'\sqrt{\tau_ie^{-\tau_i}}\longrightarrow 0 \text{  as } i\to \infty.
\end{split}
\ee
Hence we have the following convergence in Gromov-Hausdorff sense (without taking subsequence)
\be\label{level set GH conv disc}
(\Sigma_{\tau_i}, d_{\tilde{h}_i}, p_i) \longrightarrow (\{0\}, d_0, 0)  \text{  as } i \to \infty,
\ee
where $d_0$ is the discrete metric on the singleton $\{0\}$. 

To proceed, we define another type of sets similar to $M_{p, k}$ in \eqref{mpk}, namely
\be
N_{p, k}:=\{y:\, f(y)\geq f(p) -\frac{k}{\sqrt{R(p)}}\}.
\ee
We first show an analog to \eqref{ball vs mpk}: for any $k>0$
\be\label{Npk vs ball}
B_{h_i}(p_i, k)\subseteq N_{p_i, k}.
\ee
Suppose on the contrary, we can find a point $z$ $\in B_{h_i}(p_i, k)\setminus N_{p_i, k}$ and a distance minimizing geodesic $\gamma: [0, T] \longrightarrow M$ with respect to $h_i$ joining $p_i$ to $z$. By restricting $\gamma$ on a smaller interval if necessary,  
we may further assume that for all $t$, $-f(\gamma(t))\leq\tau_i+kR_i^{-1/2}$, $f(\gamma(0))=f(p_i)=-\tau_i$ and $f(\gamma(T))=-\tau_i-k R_i^{-1/2}$. Hence by $|\na_g f|_g\leq 1$,
\begin{eqnarray*}
k> d_{h_i}(p_i,z) &\geq& l_{h_i}(\gamma)\\
&=& \sqrt{R_i}\int_{0}^T |\dot{\gamma}|_g(s)\, ds\\
&\geq& -\sqrt{R_i}\int_{0}^T\la \dot{\gamma}, \na_{g} f\ra_{g}(s)\, ds\\
&=& k,
\end{eqnarray*}
which is impossible. Therefore, we must have \eqref{Npk vs ball}. To show the pointed Gromov-Hausdorff convergence, for any $\varepsilon>0$, we construct an $\varepsilon$ isometry from $F_i: N_{p_i, k} \longrightarrow [0, \infty)$ for all large $i$, 
\[ F_i(x):=\begin{cases}
      0 & \text{  if  } f(x)> -\tau_i\\
      -\sqrt{R_i}\left(f(x)+\tau_i\right)
    & \text{  if  } f(x)\leq -\tau_i.\\
   \end{cases}
\]
Obviously,  $F_i(p_i)=0$. As before, $\psi_{s}$ denote the flow of the vector field $-\frac{\na_g f}{|\na_g f|_g^2}$ with $\psi_0=\text{  id}$. For any $a,b$ $\in N_{p_i, k}$ satisfying $-\tau_i\geq f(a)\geq f(b)$, let $\b:=f(a)-f(b)\geq 0$, by the proof of \eqref{diam relation} in Proposition \ref{level sets exp decay}, 

\be\label{F dist est}
\begin{split}
d_{h_i}(a, b)&\leq d_{h_i}(a, \psi_{\b}(a))+d_{h_i}(\psi_{\b}(a), b)\\
             &\leq \sqrt{R_i}\int_0^{\b}\frac{1}{|\na_g f|_g}\,ds+\sqrt{R_i}\text{  diam  }(\Sigma_{-f(b)}, g)\\
             &\leq \sqrt{R_i}(1-ce^{-\tau_i})^{-1}\b+C\text{  diam  }(\Sigma_{\tau_i}, \tilde{h}_i)
\end{split}
\ee
Using a similar argument as in the proof of \eqref{Npk vs ball}, we also have
\be
d_{h_i}(a, b)\geq \sqrt{R_i}\b.
\ee
Hence by \eqref{diam control in GH exp}, \eqref{F dist est} and fixing $k>0$, we may take $i$ to be sufficiently large such that
\be\label{ep isom exp case est}
\begin{split}
\Big||F_i(a)-F_i(b)|-d_{h_i}(a, b)\Big|&= |\sqrt{R_i}\b-d_{h_i}(a, b)|\\
&\leq 2cke^{-\tau_i}+C\text{  diam  }(\Sigma_{\tau_i}, \tilde{h}_i) \longrightarrow 0. \\
\end{split}
\ee
When $f(a)>-\tau_i\geq f(b)$, it follows from \eqref{lev vs r} that there is a positive constant $c_0$ such that for all large $i$
\be
\{x:\, f(x)\geq-\tau_i\}\subseteq B_{g}(p_0, 2\tau_i+c_0)
\ee
and thus for all $y, z$ $\in \{x:\, f(x)\geq-\tau_i\}$,
$$d_{h_i}(y,z)\leq (4\tau_i+2c_0)\sqrt{R_i}\leq C(4\tau_i+2c_0)e^{-\tau_i/2}.$$
By \eqref{diam control in GH exp} and \eqref{ep isom exp case est}, we see that
\begin{eqnarray*}
\Big||F_i(a)-F_i(b)|-d_{h_i}(a, b)\Big|&\leq&\Big||F_i(p_i)-F_i(b)|-d_{h_i}(p_i, b)\Big|+d_{h_i}(a,p_i)\\
&\leq& 2cke^{-\tau_i}+C\text{  diam  }(\Sigma_{\tau_i}, \tilde{h}_i)\\
&&+ C(4\tau_i+2c_0)e^{-\tau_i/2} \longrightarrow 0. \\
\end{eqnarray*}
It remains to verify that $F_i$ is almost surjective, i.e. $[0, k-\varepsilon) \subseteq F_i(B_{h_i}(p_i, k))$. From the construction of $F_i$, we have for any $s$ $\in [0,k-\varepsilon)$, $s=F_i\left(\Sigma_{\tau_i+s/\sqrt{R_i}}\right)$. Thanks to \eqref{F dist est}, for all large $i$ (fixing $k$),
\begin{eqnarray*}
d_{h_i}(p_i, \Sigma_{\tau_i+s/\sqrt{R_i}}) &\leq& (1-ce^{-\tau_i})^{-1}s+C\text{  diam  }(\Sigma_{\tau_i}, \tilde{h}_i)\\
&<& k.
\end{eqnarray*}
Hence $F_i$ is an $\varepsilon$ isometry and we get the pointed Gromov-Hausdorff convergence to the ray. For the smooth convergence to a cylinder, by the local Shi derivative estimates \cite[Lemma 2.6]{Deruelle-2017}, $\na R=2\Ric(\na f)$ and \eqref{Rm by R in n dim},  we  have for all integer $k\geq 0$, there is a positive constant $C_k$ such that
\be\label{exp derivative estiamte}
|\na^k\Rm(g)|\leq C_k R\leq C'_k e^{-r}\text{  on } M. 
\ee
When restricted on $\Sigma_{\tau_0}$, $\psi_s:\Sigma_{\tau_0}\longrightarrow \Sigma_{\tau_0+s}$ are diffeomorphisms for all $s\geq 0$ and we denote the pull back metric on $\Sigma_{\tau_0}=\{-f=\tau_0\}$ by $g_s:=\psi_s^*g$. Since the second fundamental form of $\Sigma_{\tau_0+s}$ is $-\frac{\Ric}{|\na_g f|_g}$, we may apply \eqref{exp derivative estiamte} and the computation in \eqref{evolution of g in exp} to conclude that
\be
\left|\na_{g_s}^k\frac{\p}{\p s} g_s\right|_{g_s}\leq C_ke^{-\tau_0-s}
\ee
for all $s\geq 0$ and integer $k\geq 0$. By \cite[Proposition A.5]{Brendle-2010}, $g_s$ converges in $C^{\infty}$ sense to a smooth metric $g_{\infty}$ on $\Sigma_{\tau_0}$ as $s\to \infty$. The limit $g_{\infty}$ agrees with the subsequential limit in the proof of Proposition \ref{level sets are e flat} and thus is flat. Moreover, $C^{-1}g_s\leq g_{\infty}\leq Cg_s$ for all $s\geq 0$.
By the compactness of $\Sigma_{\tau_0}$ and the smooth convergence, for all $k\geq 0$, $\left|\na^k_{g_{\infty}}g_s\right|_{g_{\infty}}\leq C_k'$ for all $s\geq 0$. One may then argue by induction as in \cite[Lemma A.4]{Brendle-2010} to see that 
\be
\begin{split}
\left|\na_{g_\infty}^k\frac{\p}{\p s} g_s\right|_{g_{\infty}}&\leq C \left|\na_{g_s}^k\frac{\p}{\p s} g_s\right|_{g_s} + C\sum_{l=0}^{k-1}\left|\na_{g_\infty}^l\frac{\p}{\p s} g_s\right|_{g_{\infty}}\\
&\leq C_ke^{-\tau_0-s}.
\end{split}
\ee
Hence by integrating the above estimates, we have for all $k\geq 0$,
\be\label{level set cy estimate}
\left|\na_{g_\infty}^k \big(g_s-g_{\infty}\big)\right|_{g_{\infty}}\leq C_ke^{-\tau_0-s}.
\ee
We define the asymptotic cylinder $L$ for the steady soliton $(M,g)$ as follows. Let $L:=\R\times \Sigma_{\tau_0}$ with the product metric $g_L:=ds^2+g_{\infty}$. $\Phi: (0,\infty)\times\Sigma_{\tau_0}\longrightarrow \{-f>\tau_0\}$ is the diffeomorphism given by  $\Phi(s,\omega):=\psi_s(\omega)$. It can be seen that 
\be
\Phi^*g=|\na_g f|_g^{-2}ds^2+g_s.
\ee
Then by \eqref{exp derivative estiamte}, \eqref{level set cy estimate} and a direct (though tedious) induction argument, we have for any integers $k$, $p\geq 0$, and vectors $w_i\in T\Sigma_{\tau_0}$ with $|w_i|_{g_{\infty}}=1$,
\be\label{global cy estimate 1}
\left|\na_{g_{\infty}}^p\frac{\p^k}{\p s^k}\Big(\Phi_{(s, \omega)}^*g-g_L\Big)(w_1,\cdots, w_{p+2})\right|\leq C_{k,p}\,e^{-\tau_0-s},
\ee
where $C_{k,p}$ is some positive constant. Estimate \eqref{global cy estimate 1}
, together with the fact that $\na_{g_L}\frac{\p}{\p s}\equiv 0$ on $L$, implies the asymptotic convergence to $(L, g_L)$ at exponential rate. This completes the proof of Proposition \ref{GH limit2}.
\end{proof}

\newpage

\appendix 

\section{Dichotomy in the expanding case}
We shall give a proof of Theorem \ref{counter eg expander} which is a direct consequence of the results due to Deruelle \cite{Deruelle-2016, Deruelle-2017}. 
The key ingredient of the proof is the application of the existence and compactness results of conical expander in \cite{Deruelle-2016, Deruelle-2017}. 
\begin{proof}[\textbf{Proof of Theorem \ref{counter eg expander}:}]
We pick a smooth metric $h$ on $X=\mathbb{S}^2$ with positive but nonconstant Gauss curvature, for instance, the one induced by an ellipsoid embedded in $\R^3$. By scaling the metric $h$ if necessary, we may assume that the curvature operator of $h$ satisfies
\be\label{curvature op of h}
\Rm(h) \geq \text{id}_{\Lambda^2TX} \text{  on  } X
\ee
with equality holds somewhere at $\omega_0$ $\in X$ (this is possible since $X$ is compact and of real dimension $2$). By the existence result of conical expander \cite[Theorem 1.3]{Deruelle-2016}, there exists an asymptotically conical gradient expander $(M^3, g, f)$ with $\Rm(g)\geq 0$ and asymptotic cone given by $(C(X), dt^2+t^2h)$. Indeed, let $\{c_i\}_{i=1}^{\infty}$ be a strictly increasing sequence of positive numbers with $\lim_{i\to\infty} c_i=1$. We denote the metric $c_ih$ by $h_i$ and by \eqref{curvature op of h}
\be
\Rm(h_i) > \text{id}_{\Lambda^2TX} \text{  on  } X.
\ee
By \cite[Theorem 1.3]{Deruelle-2016}, for each $i$, there exists an asymptotically conical gradient expander $(M^3_i, g_i, f_i)$ with $\Rm(g_i)> 0$ and asymptotic cone given by $(C(X), dt^2+t^2h_i)$. Since for all large $i$, $(X, h_i)$ satisfies
\be
|\na^k\Rm(h_i)|(\omega)=c_i^{-\frac{k+2}{2}}|\na^k\Rm(h)|(\omega)\leq 2^{\frac{k+2}{2}}\sup_X|\na^k\Rm(h)|.
\ee
It follows from \cite[Remark 4.11]{Deruelle-2016} and \eqref{cone derivative estimates for metric} that we can find a sequence of positive numbers $\{\Lambda_k\}_{k=1}^{\infty}$ independent of $i$ such that
\be
\limsup_{x\to\infty} r_i^{2+k}|\na^k\Rm(g_i)|\leq \Lambda_k
\ee
and the asymptotic volume ratio is bounded from below
\be
\lim_{r\to\infty}\frac{\text{Vol}_{g_i}\left(B_{g_i}(p, r)\right)}{r^3}=\frac{\text{Vol}_{h_i}(X)}{3}\geq \frac{\text{Vol}_{h}(X)}{6}.
\ee
where $r_i$ is the distance function w.r.t. $g_i$ on $M_i$. Shifting the potential $f_i$ by constants if necessary, we apply the compactness result for conical expander by Deruelle \cite[Theorem 1.7]{Deruelle-2017} (see also \cite[Theorem 4.9]{Deruelle-2016}) and conclude that $(M^3_i, g_i, f_i)$ converges smoothly and subsequentially as $i\to\infty$ to an asymptotically conical expander $(M^3, g, f)$ with $\Rm(g)\geq 0$ and asymptotic cone given by $(C(X), dt^2+t^2h)$. Hence the existence of conical expander asserted follows. 
Using $\Rm(g)\geq 0$ and \cite[Proposition 2.4]{Deruelle-2017}, 
we have $\lim_{x\to\infty}4r^{-2}v=1$, where $v:=n/2-f$. It remains to justify \eqref{failure}. By the virtue of \eqref{cone derivative estimates for metric}, we see that $\limsup_{x\to \infty}v|\Rm|<\infty$ and for any $\omega$ $\in X$
\be\label{limit and cone}
\lim_{t\to\infty}4v\circ \phi^{-1}(t,\omega)|\Rm(g)|\circ \phi^{-1}(t,\omega)=|\Rm(g_C)|(1, \omega),
\ee
where $g_C=dt^2+t^2h$ and $\phi$ is the diffeomorphism as in Definition \ref{AC expander}. Since $g_C$ is a warped product with warping function $t$, its curvature tensor satisfies:
\begin{eqnarray}
\Rm(g_c)\left(\frac{\p}{\p t}, \cdot, \cdot, \cdot\right)&=&0\,;\\
\Rm(g_c)\left(A, B, C, D\right)(t,\omega)&=& t^2\Big[\Rm(h)\left(A, B, C, D\right)(\omega)\\
\notag&&\quad -\Big(h(A, D)h(B, C)-h(A, C)h(B, D)\Big)(\omega)\Big],
\end{eqnarray}
for any $A, B, C, D $ $\in T_{\omega}X$. Hence $\forall$ $\omega$ $\in X$,
\be
|\Rm(g_c)|(1,\omega)=\left|\Rm(h)-\frac{h\bigodot h}{2}\right|(\omega),
\ee
where $h\bigodot h_{\a\b\gamma\delta}:=2h_{\a\delta}h_{\b\gamma}-2h_{\a\gamma}h_{\b\delta}$. By the construction of $h$ \eqref{curvature op of h}, $\Rm(h)(\omega_0) = \text{id}_{\Lambda^2TX}(\omega_0)$ and \eqref{limit and cone}, it can be seen that 
\bee
0=|\Rm(g_C)|(1, \omega_0)=\liminf_{x\to \infty}4v|\Rm|. 
\eee
As $(X, h)$ is not of constant curvature and satisfies \eqref{curvature op of h}, there exists $\omega_1$ $\in X$ such that $\Rm(h)(\omega_1) > \text{id}_{\Lambda^2TX}(\omega_1)$ 
and thus both $|\Rm(g_C)|(1, \omega_1)$ and $\limsup_{x\to \infty}v|\Rm|$ are positive. This justifies \eqref{failure}.
\end{proof}


\newpage

\begin{thebibliography}{10}
\bibitem{Appleton-2017}Appleton, A., {\sl A family of non-collapsed steady Ricci solitons in even dimensions greater or equal to four}, arXiv:1708.00161 [math.DG]
\bibitem{Bamler-2020.0} Bamler, R., {\sl Entropy and heat kernel bounds on a Ricci flow background}, arXiv:2008.07093 [math.DG]
\bibitem{Bamler-2020} Bamler, R., {\sl Compactness theory of the space of super Ricci flows}, arXiv:2008.09298 [math.DG]
\bibitem{Bamler-2020.1} Bamler, R., {\sl Structure theory of non-collapsed limits of Ricci flows}, arXiv:2009.03243 [math.DG]
\bibitem{BamlerChowDengMaZhang-2021}Bamler, R.; Chow, B.; Deng, Y. X.; Ma, Z.; Zhang, Y. J., {\sl Four-Dimensional Steady Gradient Ricci Solitons with 3-Cylindrical Tangent Flows at Infinity}, arXiv:2102.04649 [math.DG]

\bibitem{Brendle-2010}Brendle, S., {\sl Ricci flow and the sphere theorem}. Graduate Studies in Mathematics, 111. American Mathematical Society, Providence, RI, 2010.
\bibitem{Brendle-2013}Brendle, S., {\sl Rotational symmetry of self-similar solutions to the Ricci flow}, Invent. Math. \textbf{194}, 731-764.
\bibitem{Brendle-2014}Brendle, S., {\sl Rotational symmetry of Ricci solitons in higher dimensions}. J. Differential Geom. 97 (2014), no. 2, 191-214. 

\bibitem{BuragoBuragoIvanov-2001} Burago, D.; Burago, Y.; Ivanov, S., {\sl A course in metric geometry}. Graduate Studies in Mathematics, 33. American Mathematical Society, Providence, RI, 2001.


\bibitem{Cao-1996} Cao, H. D., {\sl Existence of gradient K\"{a}hler Ricci solitons}, Elliptic and Parabolic Methods in Geometry (Minneapolis, MN, 1994), A K Peters, Wellesley, MA (1996), 1-16.
\bibitem{Cao-2010} Cao, H. D., {\sl Recent progress on Ricci solitons}. Recent advances in geometric analysis, 1–38, Adv. Lect. Math. (ALM), 11, Int. Press, Somerville, MA, 2010.
\bibitem{CarrilloNi-2009}Carrillo, J.; Ni, L., {\sl Sharp logarithmic Sobolev inequalities on gradient solitons and applications}, Comm. Anal. Geom. \textbf{355} (4) (2009), 721-753.
\bibitem{CatinoMastroliaMonticelli-2016}Catino, C.; Mastrolia, P.; Monticelli, D.D., {\sl Classification of expanding and steady Ricci solitons with integral curvature decay}, Geom. Topo. \textbf{20} (2016), 2665-2685.
\bibitem{Chan-2019}Chan, P.-Y., {\sl Curvature estimates for steady Ricci solitons, Trans. Amer. Math. Soc. 372 (2019)}, no. 12, 8985-9008.
\bibitem{Chan-2019.2}Chan, P.-Y., {\sl Gradient steady Kahler Ricci solitons with non-negative Ricci curvature and integrable scalar curvature}, arXiv:1908.10445 [math.DG].
\bibitem{Chan-2020} Chan, P.-Y., {\sl Curvature estimates and gap theorems for expanding Ricci solitons}, arXiv:2001.11487 [math.DG]
\bibitem{ChanMaZhang-2021} Chan, P.-Y.; Ma, Z.-L.; Zhang, Y.-J., {\sl Ancient Ricci flows with asymptotic solitons}, arXiv:2106.06904 [math.DG]


\bibitem{CheegerColding-1996} Cheeger, J.; Colding, T. H., {\sl Lower bounds on Ricci curvature and the almost rigidity of warped products}. Ann. of Math. (2) 144 (1996), no. 1, 189–237.
\bibitem{CheegerColding-1997} Cheeger, J.; Colding, T. H., {\sl On the structure of spaces with Ricci curvature bounded below. I}. J. Differential Geom. 46 (1997), no. 3, 406–480. 


\bibitem{Chen-2009}Chen, B. L., {\sl Strong Uniqueness of Ricci flow}, J. Differential Geom. \textbf{82} (2) (2009), 362-382.
\bibitem{Chicone-2006}Chicone, C., {\sl Ordinary differential equations with applications. Second edition}. Texts in Applied Mathematics, 34. Springer, New York, 2006.
\bibitem{ChowDengMa-2020}Chow, B.; Deng, Y. X.; Ma, Z.-L., {\sl On Four-dimensional Steady gradient Ricci solitons that dimension reduce}, arXiv:2009.11456 [math.DG].
\bibitem{ChowLuNi-2006}Chow, B.; Lu, P.; Ni, L., {\sl Hamilton's Ricci flow}, Graduate studies in Mathematics (2006).
\bibitem{ChowLuYang-2011}Chow, B.; Lu, P.; Yang, B., Lower bounds for the scalar curvatures of noncompact gradient Ricci solitons, C. R. Math. Acad. Sci. Paris 349 (2011), no. 23-24, 1265-1267.
\bibitem{Chowetal-2007}Chow, B. et al, {\sl The Ricci flow: techniques and applications, Part I. Geometric aspects}. Math. Survey and Monograghs, \textbf{135}, Amer. Math. Soc., Prodidence, RI, (2007).
\bibitem{Chowetal-2008}Chow, B. et al, {\sl The Ricci flow: techniques and applications. Part II. Analytic aspects}. Mathematical Surveys and Monographs, \textbf{144}. American Mathematical Society, Providence, RI, 2008
\bibitem{DengZhu-2015}Deng, Y. X.; Zhu, X. H., {\sl Complete noncompact gradient Ricci solitons with nonnegative Ricci curvature}, Math. Z. \textbf{279} (2015), 211-226.
\bibitem{DengZhu-2018.0}Deng, Y. X.; Zhu, X. H., {\sl Asymptotic behavior of positively curved steady Ricci solitons}. Trans. Amer. Math. Soc. 370 (2018), no. 4, 2855-2877.
\bibitem{DengZhu-2018}Deng, Y. X.; Zhu, X. H., {\sl Classification of gradient steady Ricci solitons with linear curvature decay}. Sci. China Math. 63 (2020), no. 1, 135-154
\bibitem{DengZhu-2019}Deng, Y. X.; Zhu, X. H., {\sl Three-dimensional steady gradient Ricci solitons with linear curvature decay}. Int. Math. Res. Not. IMRN 2019, no. 4, 1108-1124.
\bibitem{DengZhu-2020.2}Deng, Y. X.; Zhu, X. H., {\sl Higher dimensional steady Ricci solitons with linear curvature decay}, J. Eur. Math. Soc. (JEMS) 22 (2020), no. 12, 4097-4120.
\bibitem{DengZhu-2020.3}Deng, Y. X.; Zhu, X. H., {\sl Rigidity of $\kappa$-noncollapsed steady K\"{a}hler-Ricci solitons}. Math. Ann. 377 (2020), no. 1-2, 847-861
\bibitem{Deruelle-2012}Deruelle, A., {\sl Steady gradient Ricci soliton with curvature in $L^1$}. Comm. Anal. Geom. 20 (2012), no. 1, 31-53.
\bibitem{Deruelle-2016}Deruelle, A., {\sl Smoothing out positively curved metric cones by Ricci expanders}. Geom. Funct. Anal. 26 (2016), no. 1, 188–249.
\bibitem{Deruelle-2017}Deruelle, A., {\sl Asymptotic estimates and compactness of expanding gradient Ricci solitons}, Ann. Sc. Norm. Super. Pisa Cl. Sci. (5) 17 (2017), no. 2, 485-530.
\bibitem{DeTurckKazdan-1981}DeTurck, D.; Kazdan, J., {\sl Some regularity theorems in Riemannian geometry}. Ann. Sci. École Norm. Sup. (4) 14 (1981), no. 3, 249-260.
\bibitem{FeldmanIlmanenKnopf}Feldman, F.; Ilmanen, T.; Knopf, D., {\sl Rotationally symmetric shrinking and expanding gradient K\"{a}hler Ricci solitons}, J. Differential Geom. \textbf{65} (2003), 169-209.

\bibitem{FM-2008} Fern\'{a}ndez-L\'{o}pez, M.; Garc\'{i}a-R\'{i}o, E., {\sl A remark on compact Ricci solitons}. Math. Ann. 340 (2008), no. 4, 893–896.

\bibitem{GilbargTrudinger-2001}Gilbarg, D., Trudinger, N., {\sl Elliptic partial differential equations of Second order, reprint of the 1998 edition}, Springer (2001).
\bibitem{Gromov-1978}Gromov, M., {\sl Almost flat manifolds}. J. Differential Geometry 13 (1978), no. 2, 231-241.
\bibitem{GL-spin} Gromov, M.; Lawson B., {\sl Spin and Scalar Curvature in the Presence of a Fundamental Group,} I. Annals of Mathematics, 111(2), 209-230. doi:10.2307/1971198.
\bibitem{Hamilton-1995}Hamilton, R., {\sl The formation of singularities in the Ricci flow}, Surveys in Differential Geometry \textbf{2} (1995), 7-136, International Press.
\bibitem{Hamilton-1995.2}Hamilton, R., {\sl A Compactness Property for Solutions of the Ricci Flow}. American Journal of Mathematics, 117(3) (1995), 545-572.
\bibitem{Ivey-1996}Ivey T., {\sl Local existence of Ricci solitons}. Manuscripta Math. 91 (1996), no. 2, 151-162.
\bibitem{Kotschwar-2013}Kotschwar, B., {\sl A local version of Bando's theorem on the real-analyticity of solutions to the Ricci flow}, Bulletin of the London Mathematical Society 45(1) (2013), 153-158.
\bibitem{KotschwarWang-2015}Kotschwar, B.; Wang, L., {\sl Rigidity of asymptotically conical shrinking gradient Ricci solitons}. J. Differential Geom. 100 (2015), no. 1, 55–108.
\bibitem{KobayashiNomizu-1963}Kobayashi, S.; Nomizu, K., {\sl Foundations of Differential Geometry, Vol. 1}, Wiley and Sons, New York, 1963.
\bibitem{KobayashiNomizu-1969}Kobayashi, S.; Nomizu, K., {\sl Foundations of Differential Geometry, Vol. 2}, Wiley and Sons, New York, 1969.
\bibitem{Lai-2020} Lai, Y., {\sl A family of 3d steady gradient solitons that are flying wings}, arXiv:2010.07272 [math.DG].
\bibitem{MaZhang-2021} Ma, Z.-L.; Zhang, Y.-J., {\sl Perelman's entropy on ancient Ricci flows}, arXiv:2101.01233 [math.DG].
\bibitem{MunteanuSesum-2013} Munteanu, O.; Sesum, N., {\sl On gradient Ricci solitons}. J. Geom. Anal. 23 (2013), no. 2, 539-561.
\bibitem{MunteanuSungWang-2017}Munteanu, O.; Sung, C. J. A.; Wang, J. P., {\sl Poisson equation on complete manifolds},  Adv. Math. \textbf{348} (2019), 81-145.
\bibitem{MunteanuWang-2011}Munteanu, O.; Wang, J. P., {\sl Smooth metric measure spaces with non-negative curvature}. Comm. Anal. Geom. 19 (2011), no. 3, 451-486.
\bibitem{MunteanuWang-2017}Munteanu, O.; Wang, J. P., {\sl Conical structure for shrinking Ricci solitons}. J. Eur. Math. Soc. (JEMS) 19 (2017), no. 11, 3377–3390.
\bibitem{MunteanuWang-2019} Munteanu, O.; Wang, J. P., {\sl Structure at infinity for shrinking Ricci solitons}. Ann. Sci. \'{E}c. Norm. Sup\'{e}r. (4) 52 (2019), no. 4, 891-925.
\bibitem{p1}
Perelman, G., {The entropy formula for ricci flow and its geometric applications.} {arXiv preprint math/0211159, 2002}.

\bibitem{p2}
Perelman, G., {Ricci flow with surgery on three-manifolds.} arXiv preprint
math/0303109, 2002.
\bibitem{p3}
Perelman, G., {Finite extinction time for the solutions to the ricci flow on
certain three-manifolds}. arXiv preprint math/0307245, 2003.

\bibitem{Petersen-2016} Petersen, P., {\sl Riemannian geometry. Third edition}. Graduate Texts in Mathematics, 171. Springer, Cham, 2016.

\bibitem{SY} Schoen, R.; Yau, S. T., {\sl On the structure of manifolds with positive scalar curvature,} Manuscripta Math 28, 159–183 (1979). https://doi.org/10.1007/BF01647970.




\bibitem{Zhang-2009}Zhang, Z. H., {\sl On the completeness of gradient Ricci solitons}, Proc. Amer. Math. Soc. \textbf{137} (2009), 2755-2759.
\end{thebibliography}
\end{document}